\documentclass[a4paper,12pt]{article}
\usepackage[utf8]{inputenc}
\usepackage{enumerate}
\usepackage{amssymb, bm}
\usepackage{amsthm}
\usepackage{amsmath}
\usepackage{tikz-cd}
\usepackage{tikz}
\usepackage{ulem}
\usepackage{bbold}
 \usepackage[all]{xy}
 
\newtheorem{mythm}{Theorem}[section]
\newtheorem{mylem}[mythm]{Lemma}
\newtheorem{myprop}[mythm]{Proposition}

\newtheorem{myex}[mythm]{Example}
\newtheorem{mydef}[mythm]{Definition}
\newtheorem{myrem}[mythm]{Remark}

\title{On neighborhoods of embedded toroidal and Hopf manifolds and their foliations}
\author{Laurent Stolovitch and Xiaojun Wu}

\newcommand{\diag}{\operatorname{diag}}

	\newcommand{\Om}{\Omega}
		\newcommand{\ov}{\overline}

\def\Z{\mathbb{Z}}
\def\Q{\mathbb{Q}}  \def\C{\mathbb{C}}
 \def\R{\mathbb{R}}
 \def\N{\mathbb{N}}

  \def\d{\partial}
 \def\dbar{{\overline{\partial}}}

\def\cA{\mathcal{A}}
  \def\cO{\mathcal{O}}

\def\cU{\mathcal{U}}

\def \Im{\mathrm{Im}}
\def \Re{\mathrm{Re}}

\newcommand{\dindice}[2]{{\stackrel{\scriptstyle #1}{\scriptstyle #2}}}

\newcommand{\cL}{\mathcal}

\newcommand{\var}{\varphi}

\begin{document}
\maketitle
\begin{abstract}
In this article, we give completely new examples of embedded complex manifolds the germ of neighborhood of which is holomorphically equivalent to a germ of neighborhood of the zero section in its normal bundle. The first set of  examples is composed of connected abelian complex Lie groups, embedded in some complex manifold $M$. These are non compact manifolds in general. We also give some conditions ensuring the existence a holomorphic foliation having the embedded manifold as leaf. 
The second set of examples are $n$-dimensional Hopf manifolds, embedded as hypersurfaces.
\end{abstract}
\section{Introduction}

It's  a classical problem to classify neighborhoods $U$ of an embedded compact complex manifold $C$ into complex manifolds $M$ up to biholomorphism fixing $C$ point wise. In particular, Grauert called ``Das Formale Prinzip''(e.g \cite{kosarew-crelle,Hi81,hwang-annals}) the following problem : Assume there is  a {\it formal equivalence} between a neighborhood $U$ in $M$ and a neighborhood $U'$ in $M'$, in both of which $C$ is holomorphically embedded. Does such a formal equivalence give  rise to a genuine holomorphic equivalence between the (possibly smaller) neighborhoods? Here differential geometry and curvature enter into the play. Indeed, if this normal bundle is {\it negative}, then Grauert~\cite{grauert-embed} and Hironaka-Rossi~\cite{hironaka-rossi} 
proved a rigidity statement: formally equivalent neighborhoods of $C$ in $M$ and $M'$ are actually biholomorphic. In the case the normal bundle is {\it positive}, Griffiths~\cite{griffiths-ext2} proved that there are only
finitely many obstructions to being formally equivalent to $U$. He then
proved, under some assumptions, that if the neighborhood $U'$ is formally
equivalent to $U$, then it is also biholomorphic to it (see \cite{GS-nf} for moduli spaces). 

A holomorphic embedding of $C$ into $M$ gives rise to another natural embedding, namely the embedding of $C$ as the zero section in its normal bundle $T_{C|M}$. Equivalence of a neighborhood of $C$ in $M$ with a neighborhood of $C$ in $T_{C|M}$ can be seen as kind of "linearization". In that case, we shall say that the neighborhood is "fully linearizable". When the normal bundle is {\it flat},  dynamical systems methods are more appropriate. Indeed, Arnol'd~\cite{arnold-embed}(see also \cite{arnold-geometrical}[Chap.~5, sect.~27]) studied the
embedding of an elliptic curve into a complex surface when the normal bundle has
zero self-intersection number (i.e. {\it flat}). He showed that under a {\it small divisors condition}, the neighborhood is biholomorphic to a (unspecific) neighborhood of the zero section of the normal bundle $N_{C|M}$. This was generalized by Ilyashenko and Pyartly~\cite{ilyashenko-pyartly-embed} to direct product of 1 dimensional tori.

Another somehow similar problem, following Ueda \cite{Ued82}, is the existence of an holomorphic foliation in a neighborhood of $C$, having $C$ as a leaf and can be obtained as through a "vertical linearization" of neighborhoods. 

In the recent year, there has a renewal of interest in these questions (e.g. \cite{loray-moscou,stolo-koike}), in particular in the case of flat normal bundle \cite{GS21,koike-ueda,koike-ueda0, stolo-koike}. 

In this article, we shall answer to these questions for some special (possibly non-compact) embedded manifolds, namely for connected abelian complex Lie group. Our main results are Theorem \ref{ueda-thm} and Theorem \ref{full-toroidal}.
 One can easily show that every any connected abelian complex Lie group is isomorphic to complex Lie group to $\C^n/ \Lambda$ where $\Lambda$ is a discrete subgroup of $\C^n$.
	It is compact if and only if it is a torus.
These groups are the simplest complex Lie groups.
According to \cite{abe-kopferman01}, toroidal groups (as the non-trivial component of connected abelian complex Lie group) are the missing link between torus groups and any complex Lie groups (in Lie group theory), between compact torus groups and Stein groups (in several complex variables).
The first non-trivial example was found when Pierre Cousin studied meromorphic functions of two variables with triple periods in \cite{Cou10} and discovered that such a group does not contain $\C$ or $\C^*$ as direct summand.
Kopfermann studied systematically toroidal groups in \cite{Kop64} and Morimoto proved a fundamental holomorphic reduction theorem for complex Lie groups in \cite{Mor65}.
We refer to the book \cite{abe-kopferman01} for further information and more recent results.

Our recent work related to embedded tori \cite{GS,stolo-wu-ueda} will serve as a guideline to develop the required technics.

On the other hand, Hopf manifolds are compact complex manifolds the embedding of which has been studied in the cases of codimension-$1$ surfaces \cite{Tsu84}. 
 
 We shall use our recent work \cite{GS21} to show that $n$-dimensional general Hopf manifolds holomorphically embedded as hypersurface have a germ of neighborhood holomorphically equivalent to a germ of neighborhood of the zero section in its normal bundle (see Theorem \ref{hopf}).

\section{Toroidal manifolds}
\subsection{Preliminaries}
We first recall some properties of abelian complex Lie groups from \cite{abe-kopferman01}.
The following class of groups is of particular importance:
\begin{mydef}
An abelian complex Lie group $F$ is called a toroidal group if $H^0(F, \cO_F)=\C$, i.e.,   if all holomorphic functions are constant.
\end{mydef}
Consider $T$ a connected abelian complex Lie group of complex dimension $n$.
We have the following decomposition theorem due to Remmert-Morimoto (e.g \cite{Mor65}, \cite{abe-kopferman01}[Theorem 1.1.5])~:
\begin{mythm}
Every connected abelian complex Lie group is holomorphically isomorphic to a
$$\C^a \times (\C^*)^{b} \times X_0$$
with a toroidal group $X_0$.
The decomposition is unique.
Moreover, the group is Stein if and only if the toroidal group component in the above decomposition is trivial.
\end{mythm}

We study the toroidal group part $X_0$ (of complex dimension $n-a-b$) in more details denoted in the following by $T$.

By \cite[Sect. 1]{abe-kopferman01} one can choose a basis
 $P\in \mathrm{Mat}(q, 2q, \C)$  the period matrix of a compact complex torus and $R \in \mathrm{Mat}(n-q-a-b, 2q, \R)$ is a real matrix, called gluing matrix, satisfying the irrationality condition
\[ \forall \sigma \in \Z^{n-q-a-b}\colon \sigma^tR\not \in \Z^{2q}\]
such that the given toroidal group $T$ satisfies
\[ T = \C^{n-a-b} / \begin{pmatrix}
 I_{n-q-a-b} & R\\
 0& P\end{pmatrix}\Z^{n+q-a-b}.\]
Here $q$ is called the rank of $T$. The above choice of basis induces (by projection on the last $q$ variables) a $(\C^*)^{n-q-a-b}$-principal bundle structure $T\to T_0 =  \C^q/P\Z^{2q}$ over a compact complex torus.
We may also assume that
$$P=(I_q,P_0)$$
for some $P_0 \in \mathrm{Mat}(q, q, \C)$
such that $P_0$ is the period matrix of $T_0$.
Write
$$R=(R_1,R_2)$$
such that $T$ admits a covering
\begin{equation}
 (\C^*)^{n-a-b} \cong \C^{n-a-b} / \begin{pmatrix}
 I_{n-q-a-b} & R_1\\
 0& I_q \end{pmatrix}\Z^{n-a-b}\label{cover}
\end{equation}

 where the identification is given by

 \begin{align*}
(z_i) \in \C^{n-a-b} / \begin{pmatrix}
 I_{n-q-a-b} & R_1\\
 0& I_q \end{pmatrix}\Z^{n-a-b} \mapsto &\\
 (e^{2 \sqrt{-1} \pi(z_1-\sum_{j=1}^q(R_1)_{1,j}z_{n-q-a-b+j})}, \cdots,e^{2 \sqrt{-1} \pi(z_{n-q-a-b}-\sum_{j=1}^q(R_1)_{n-q-a-b,j}z_{n-q-a-b+j})},&\\
 	e^{2 \sqrt{-1} \pi z_{n-q-a-b+1}}, \cdots, e^{2 \sqrt{-1} \pi z_{n-a-b}} ).&
 \end{align*}
 Recall the real parametrization following \cite[1.1.13]{abe-kopferman01} from which we construct a fundamental domain.
 Consider the $n+q-a-b$ columns of the matrix
 $$
 \begin{pmatrix}
 I_{n-q-a-b} & R\\
 0& P\end{pmatrix}
$$
which are $\R-$linearly independent.
Take complements to get a $\R-$linear basis $\gamma_i(i \leq 2(n-a-b))$
whose coordinates in matrix correspond to
 $$
 \begin{pmatrix}
 I_{n-q-a-b} & R_1 & R_2 & R_3\\
 0& I_q & P_0 & P_1\end{pmatrix}=[\gamma_1,\ldots,\gamma_{2(n-a-b)}].
$$
Thus we have diffeomorphism
$$(\R/\Z)^{n+q-a-b} \times \R^{n-q-a-b} \to T,$$
$$(t_i) \mapsto \sum t_i \gamma_i.$$
Thus a fundamental domain is
$$\omega_0=\{\sum_{1 \leq i \leq n+q-a-b} t_i\gamma_i+\sum_{ n+1+q-a-b\leq j \leq 2(n-a-b) } \R \gamma_j \in \C^{n-a-b}, t_i \in [0,1[  \}.$$
Let us set 
\begin{eqnarray}
	\mathfrak{z}_{\ell}(z)&:=& Z_{\ell}-\sum_{j=1}^q(R_1)_{\ell,j}Z_{n-q-a-b+j},\quad 1\leq\ell\leq n-q-a-b\label{frakz}\\
	\mathfrak{z}_{\ell}(z)&:=& Z_{\ell},\quad n-q-a-b+1\leq \ell\leq n-a-b.\nonumber
\end{eqnarray}

The variables $\mathfrak{z}_{\ell}$ are called standard coordinates of the toroidal group $T$ (cf. \cite{abe-kopferman01}[1.1.11]).
The previous chosen coordinates are called toroidal coordinates (cf. \cite{abe-kopferman01}[1.1.12]).
Under the standard coordinates, the lattice defined by $\gamma_i(1 \leq i \leq 2(n-a-b))$ of $\C^{n-a-b} \simeq \R^{2(n-a-b)}$ corresponds to the columns of the matrix
$$
\begin{pmatrix}
 I_{n-q-a-b} & 0 & R_2-R_1P_0 & R_3-R_1P_1\\
 0& I_q & P_0 & P_1\end{pmatrix}.
 $$
 In the following, we use symbols $\omega_\bullet$ to define domains in toroidal coordinates and symbols $\Omega_\bullet$ to define domains in standard coordinates.
 The choice of standard coordinates is for the propose that the domains $\Omega_\bullet$ are thus Reinhardt (i.e. a domain such that $$ (e^{\sqrt{-1} \theta_1} z_1, \cdots , e^{\sqrt{-1} \theta_n} z_n, e^{\sqrt{-1} \theta_{n+1}} v_1, \cdots, e^{\sqrt{-1} \theta_{n+d}} v_d)\in\Omega_{\bullet}$$ for
every $z = (z_1, \cdots , z_n, v_1, \cdots, v_d) \in \Omega_{\bullet}$ and $\theta_1, \cdots \theta_n, \theta_{n+1}, \cdots, \theta_{n+d} \in \R$). Consider
$$\Omega_0=\{(e^{2 \sqrt{-1} \pi\mathfrak{z}_1},\cdots, e^{2 \sqrt{-1} \pi\mathfrak{z}_{n-a-b}}),\quad (Z_1, \cdots, Z_{n-a-b} ) \in \omega_0\}.$$
with
\begin{align*}
\omega_0&=\left\{ 
\begin{pmatrix}
	I_{n-q-a-b} & R_1 & R_2 & R_3\\
	0& I_q & P_0 & P_1\end{pmatrix}(w_1, \cdots, w_{2(n-a-b)} )^t,\right.\\
&\left.w_i \in [0,1[, 1 \leq i \leq n+q-a-b, w_j \in \R,  j \geq n+q-a-b+1\right\}.
\end{align*}
In other words,
$$\Omega_0=\{(e^{2 \pi \sqrt{-1} \mathfrak{z}_1}, \cdots, e^{2 \pi \sqrt{-1} \mathfrak{z}_{n-a-b}}),$$$$  (\mathfrak{z}_1, \cdots, \mathfrak{z}_{n-a-b} )^t= \begin{pmatrix}
 I_{n-q-a-b} & 0 & R_2-R_1P_0 & R_3-R_1P_1\\
 0& I_q & P_0 & P_1\end{pmatrix} (w_1, \cdots, w_{2(n-a-b)} )^t,
 $$$$w_i \in [0,1[, 1 \leq i \leq n+q-a-b, w_j \in \R,  j \geq n+q-a-b+1   \}.$$
Thus $\Omega_0$ is a Reinhardt domain.

Define the numbers $\tau_{i,j}(1 \leq i \leq n-a-b,1 \leq j \leq q)$ as the $(i,j)-$th component of the matrix
$$
\begin{pmatrix}
 R_2-R_1P_0 \\
  P_0 \end{pmatrix}=[\gamma'_1, \cdots,\gamma'_q].
$$

Consider the following Reinhardt domains
$$\omega_\epsilon=\{\sum_{1 \leq i \leq n-a-b} t_i\gamma_i+\sum_{n-a-b \leq j \leq n+q-a-b} s_j\gamma_j+\sum_{ n+1+q-a-b\leq h \leq 2(n-a-b) } \R \gamma_h \in \C^{n-a-b},$$$$ t_i \in [0,1[,\quad s_j \in ]-\epsilon,1+\epsilon[  \}.$$
$$\Omega_\epsilon=\{(e^{2 \sqrt{-1} \pi\mathfrak{z}_1},\cdots, e^{2 \sqrt{-1} \pi\mathfrak{z}_{n-a-b}}),\quad (Z_1, \cdots, Z_{n-a-b} ) \in \omega_{\epsilon}\}.$$
In other words,
$$\Omega_\epsilon=\{(e^{2 \sqrt{-1} \pi\mathfrak{z}_1}, \cdots, e^{2 \sqrt{-1} \pi\mathfrak{z}_{n-a-b}}),$$
$$  (\mathfrak{z}_1, \cdots, \mathfrak{z}_{n-a-b} )^t= \begin{pmatrix}
 I_{n-q-a-b} & 0 & R_2-R_1P_0 & R_3-R_1P_1\\
 0& I_q & P_0 & P_1\end{pmatrix} (w_1, \cdots, w_{2(n-a-b)} )^t,
 $$$$w_i \in [0,1[, 1 \leq i \leq n-a-b,  w_j \in ]-\epsilon,1+\epsilon[,n-a-b+1 \leq  j \leq n+q-a-b,$$$$ w_h \in \R, h \geq n+q-a-b+1   \}.$$
$$\omega_{\epsilon,R}=\{\sum_{1 \leq i \leq n-a-b} t_i\gamma_i+\sum_{n-a-b \leq j \leq n+q-a-b} s_j\gamma_j+\sum_{ n+1+q-a-b\leq h \leq 2(n-a-b) } r_h \gamma_h \in \C^{n-a-b},$$$$ t_i \in [0,1[, s_j \in ]-\epsilon,1+\epsilon[, r_h \in ]-R,R[  \}.$$
Consider also the following Reinhardt domains
$$\Omega_{\epsilon,R}=\{(e^{2 \sqrt{-1} \pi\mathfrak{z}_1},\cdots, e^{2 \sqrt{-1} \pi\mathfrak{z}_{n-a-b}}),\quad (Z_1, \cdots, Z_{n-a-b} ) \in \omega_{\epsilon,R}\}$$
In other words,
$$\Omega_{\epsilon,R}=\{(e^{2 \pi \sqrt{-1} \mathfrak{z}_1}, \cdots, e^{2 \pi \sqrt{-1} \mathfrak{z}_{n-a-b}}),$$$$  (\mathfrak{z}_1, \cdots, \mathfrak{z}_{n-a-b} )^t= \begin{pmatrix}
 I_{n-q-a-b} & 0 & R_2-R_1P_0 & R_3-R_1P_1\\
 0& I_q & P_0 & P_1\end{pmatrix} (w_1, \cdots, w_{2(n-a-b)} )^t,
 $$$$w_i \in [0,1[, 1 \leq i \leq n-a-b,  w_j \in ]-\epsilon,1+\epsilon[,n-a-b+1 \leq  j \leq n+q-a-b,$$$$ w_h \in ]-R,R[, h \geq n+q-a-b+1   \}.$$
 In particular, $\Omega_{\epsilon,R}$ is relatively compact in $(\C^*)^{(n-a-b)}$.
 Note that as Reinhardt domains, any holomorphic function on any of the above domains admits a unique Laurent series.
We have that
$$\cup_{R \in \R^+} \Omega_{\epsilon,R}=\Omega_{\epsilon}. $$
\begin{myex}
A basic example of toroidal group is the following.
$$T=\C^{2} / \begin{pmatrix}
 1 & a & b\\
 0& 1 & \sqrt{-1} \end{pmatrix}\Z^{3}$$
 such that $(a \Z+b\Z) \cap \Q=0$ with $a,b \in \R$.
 After a shear transformation (from toroidal coordinates), we have in standard coordinates,
 $$T=\C^{2} / \begin{pmatrix}
 1 & 0 & b-a \sqrt{-1}\\
 0& 1 & \sqrt{-1} \end{pmatrix}\Z^{3}.$$
 Take a $\R-$basis corresponding to
 $$ \begin{pmatrix}
 1 & 0 & b-a \sqrt{-1} & \sqrt{-1}\\
 0& 1 & \sqrt{-1}&0 \end{pmatrix}.$$
 Then $$\Omega_\epsilon=\{(e^{2 \pi \sqrt{-1} (t_1+(b-a\sqrt{-1})t_3)+\sqrt{-1}t_4}, e^{2 \pi \sqrt{-1} (t_2+\sqrt{-1}t_3)}),$$$$t_1 \in [0,1[,t_2 \in [0,1[, t_3 \in ]-\epsilon,1+\epsilon[,t_4 \in \R\},$$
 $$\Omega_{\epsilon,R}=\{(e^{2 \pi \sqrt{-1} (t_1+(b-a\sqrt{-1})t_3)+\sqrt{-1}t_4}, e^{2 \pi \sqrt{-1} (t_2+\sqrt{-1}t_3)}),$$$$t_1 \in [0,1[,t_2 \in [0,1[, t_3 \in ]-\epsilon,1+\epsilon[,t_4 \in ]-R,R[\}.$$
\end{myex}
For the convenience of the readers, recall the following easy lemma.
\begin{mylem}
Let $\Omega_i(i \in \N)$ be a family of Reinhardt domains in some fixed $\C^N$.
Let $f_i \in \cO(\Omega_i)$ such that $f_i$ have the same Laurent power series.
Then on any $i,j$, $f_i=f_j$ on $\Omega_i \cap \Omega_j$ which induces a holomorphic function on $\cup_i \Omega_i$.
\end{mylem}
In the following, we will choose suitable conditions such that all the cohomological operators have a unique solution in the ring of Laurent series.
Thus to show the existence of the vertical/full linearization,
it is enough to show the convergence of Laurent series for each Reinhardt domain in a well-chosen family of Reinhardt domains whose union is an open neighborhood of the given connected abelian complex Lie group.
\subsection{Preliminary lemmas}
In what follows, we give a non-compact version of some results of \cite{GS}, the proofs of which are identical.
The following lemma is a verbatim adaptation of \cite[Lemma 4.1]{GS} which relates the covering of the submanifold and the covering of its neighborhood.
\begin{mylem}
Let $C$ be a complex manifold. Let $\pi\colon \tilde{C}\to C$ be a holomorphic covering and $\pi(x_0^*)=x_0$. Suppose that $(M,C)$ is a holomorphic neighborhood of $C$. There is a neighborhood $U$ in $M$ of $C$ 
 and  a holomorphic neighborhood  $\tilde U$ of $\tilde{C}$ such that $p\colon\tilde U\to U$  is an extended covering of the covering $\pi\colon \tilde{C}\to C$ and $C$
(resp. $\tilde{C}$)  is a smooth strong retract of $U$ (resp. $\tilde U$). Consequently, $$
  \pi_1(\tilde U,x_0^*)=\pi_1(\tilde{C},x_0^*), \quad \pi_1(U,x_0)=\pi_1(C,x_0).
$$
\end{mylem}
Applying the above lemma to $(N_C,C)$ with $C$ a connected abelian complex Lie group and a covering $\pi|_{\tilde{C}}\colon \tilde{C}\to C$ such that $\tilde{C}$ is a product of several copies of $\C$ and cylinders $\C/\Z\simeq\C^*$, we have a covering $\hat\pi\colon \widetilde{N_C}\to N_C$ such that
$$
\tilde{C}\subset\widetilde{N_C}, \quad \pi_1(\widetilde{N_C},x_0^*)=\pi_1(\tilde{C},x_0^*), \quad \pi_1(N_C,x_0)=\pi_1(C,x_0).
$$
The following result is a verbatim adaptation of \cite{GS}[Proposition 4.3] on the classification of pair $(C,M)$~:
\begin{myprop}
Let $C$ be the connected abelian complex Lie group and $\pi \colon\tilde{C}=\C^{n}/{\Z^{n-a}}\to C$ be the covering constructed in (\ref{cover}).
Let $(M,C)$ be a neighborhood of $C$. Assume that $N_C$ admits transition functions that
are locally constant and diagonalizable matrices (e.g. Hermitian flat).
Then
$(M,C)$ is holomorphically equivalent to the quotient space of an open neighborhood of $\tilde{C}$ in $\widetilde{N_C}$  by the Deck transform of $\widetilde{N_C}$.
Moreover,
one can take $\hat\omega_{\epsilon_0}$
some Stein neighborhood of
$$\C^a \times (\C^*)^{b} \times \omega_{\epsilon_0} \times \{0\}$$
(for suitable choice of $\epsilon_0$)
such that $(M,C)$ is biholomorphic to the quotient of $\hat\omega_{\epsilon_0}$ by  $\tau^0_1,\dots,\tau^0_q$.
Here for any $1 \leq i \leq q$, $\tau^0_i$
corresponds to
translation by $ \gamma_{i+n-a-b}$ (in toroidal coordinates) when restricting on $\C^a \times (\C^*)^{b} \times \omega_\epsilon \times \{0\}$.

An equivalent way in standard coordinates can be reformulated as follows~:
Recalling notation (\ref{frakz}), let $$\Omega_{\epsilon_0}:=\{(e^{2 \sqrt{-1}\pi \mathfrak{z}_1},\cdots, e^{2 \sqrt{-1}\pi  \mathfrak{z}_{n-a-b}}),\quad (Z_1, \cdots, Z_{n-a-b} ) \in \omega_{\epsilon_0}\}.$$ 
Let $\tau_j$ be the mapping defined on 
some Stein neighborhood $\hat\Omega_{\epsilon_0}$ of
$$\C^a \times (\C^*)^{b} \times \Omega_{\epsilon_0} \times \{0\}$$
such that the restriction of $\tau_j$ on $\C^a \times (\C^*)^{b} \times \Omega_{\epsilon_0} \times \{0\}$
corresponds to translations by $\gamma'_j$ in standard coordinates (before taking their exponentials).
Then $\tau_1,\dots, \tau_q$ commute pairwise   wherever they are defined, i.e.
$$
\tau_i\tau_j(x,y,h,v)=\tau_j\tau_i(x,y,h,v)\quad  \forall  i\neq j
$$
for $(x,y,h,v)\in \hat\Omega_{\epsilon_0}\cap \tau_i^{-1}\hat\Omega_{\epsilon_0}\cap \tau_j^{-1}\hat\Omega_{\epsilon_0}$.
Notice that
$(M,C)$ is also biholomorphic to the quotient of $\hat\Omega_{\epsilon_0}$ by  $\tau_1,\dots,\tau_q$.

On the other hand, let $(\tilde M,C)$ be another  such (not necessarily Stein but open) neighborhood having the corresponding generators $\tilde\tau_1,\dots,\tilde\tau_n$  defined on $\hat\Omega_{\tilde\epsilon_0}$. Then $(M,C)$ and $(\tilde M,C)$ are holomorphically equivalent if and only if there is a  biholomorphic mapping $F$  from $\hat\Omega_{\epsilon}$ into $\hat\Omega_{\tilde \epsilon}$
for some positive $\epsilon,\tilde \epsilon$  such that
$$
F \tilde\tau_j(x,y,h,v)=\tau_jF(x,y,h,v),\ j=1,\ldots, n,
$$
wherever   both sides are defined, i.e. $(x,y,h,v)\in \hat\Omega_{\tilde\epsilon}\cap \tilde\tau_j^{-1}\hat\Omega_{\epsilon}\cap \hat\Omega_{\epsilon}\cap F^{-1}\hat\Omega_{\epsilon}.$
\end{myprop}
As in \cite{GS}[(4.7), (4.8)],  the deck transformations of $\widetilde{N_C}$ in this case can be written
as follows~: for any $1 \leq j \leq q$, $\forall (x,y, h',v') \in \C^a \times (\C^*)^{b} \times (\C^*)^{n-a-b} \times \C^d$~:
\begin{equation} \hat \tau_j(x,y,
 	h',v')
 	=
 	(x,y,T_jh',  M_jv') \label{tauhat},\quad \hat\tau_0(x,y,h',v'):=(x,y,h',v')
\end{equation}
for some diagonal matrix
 $$
 T_j:= \diag(\lambda_{j,1},\dots, \lambda_{j,n-a-b}),\quad M_j:=\diag(\mu_{j,1},\dots, \mu_{j,d}),\quad \lambda_{j,i}=e^{2 \pi \sqrt{-1} \tau_{j,i}}.$$
The proof is the same as \cite{GS}[(4.7), (4.8)] by noticing that toroidal group has some quotient space which is a complex torus. As the deck transformations $\tau_i$'s of $(M,C)$ are nonlinear perturbations (wrt to vertical variables) of the $\hat\tau_i$'s, the previous proposition says that we need to simultaneously linearize the $\tau_i$'s to show that $(M,C)$ is holomorphically equivalent to $(N_C,C)$. 
\subsection{Spaces of functions}
In the following, we adapt the following notation for any decreasing function $r: \R_{>0} \to \R_{>0}$,
\begin{align*}
\Omega_{\epsilon_0,R, r(R)}=& D(0,R)^a \times (D(0,R) \setminus D(0, \frac{1}{R}))^{b} \times \Omega_{\epsilon_0,R} \times D(0, r(R))^d,\\
\Omega_{\epsilon_0,r}=&\cup_{R>0} \Omega_{\epsilon_0,R, r(R)}
\end{align*}
where $D(0,R)$ is the disc centered at the origin with radius $R$.
Unlike the torus case, we need to introduce the following family of semi norms instead of a single norm.
\begin{mydef}
\label{def-domains}
Fix a decreasing function $r: \R_{>0} \to \R_{>0}$.
Set
$
\tilde \Omega_{\epsilon,R,r(R)}:=\overline\Omega_{\epsilon,R,r(R)}\cup\bigcup_{i=1}^n (T_i, M_i) (\overline\Omega_{\epsilon,R,r(R)})$,
$
\tilde \Omega'_{\epsilon,R,r(R)}:=\overline\Omega_{\epsilon,R,r(R)}\cup\bigcup_{i=1}^n (T_i, M_i)^{-1} (\overline\Omega_{\epsilon,R,r(R)})$.
Denote by $\cA_{\epsilon,R,r}$
$($resp. $\tilde{\cA}_{\epsilon,R, r}$,  $\tilde{\cA'}_{\epsilon,R, r})$)
the set of holomorphic functions on $\overline{\Omega_{\epsilon,R,r(R)}}$,
$($resp. $\overline{\tilde \Omega_{\epsilon,R,r(R)}}$, $\overline{\tilde \Omega'_{\epsilon,R,r(R)}})$.
If $f\in \mathcal{A}_{\epsilon,R, r}$,
$($resp. $\tilde f\in\tilde{\cA}_{\epsilon,R, r}$,
resp. $\tilde f' \in\tilde{\cA'}_{\epsilon,R, r})$)
we set
$$
||f||_{\epsilon,R,r}:=\sup_{(x,y,h,v)\in\Omega_{\epsilon,R,r(R)}}|  f(x,y,h,v)|,\quad |||\tilde f|||_{\epsilon,R,r}:=\sup_{(x,y,h,v)\in\tilde\Omega_{\epsilon,R,r(R)}}| \tilde f(x,y,h,v)|,
$$
$$
|||\tilde f'|||'_{\epsilon,R,r}:=\sup_{(x,y,h,v)\in\tilde\Omega'_{\epsilon,R,r(R)}}| \tilde f'(x,y,h,v)|.
$$
Define
$\mathcal{A}_{\epsilon, r}$ (resp.$\tilde{\cA}_{\epsilon, r}$, $\tilde{\cA'}_{\epsilon, r}$ )
 as the inductive limit of
 $\mathcal{A}_{\epsilon,R, r}$ (resp. $\tilde{\cA}_{\epsilon,R, r}$, $\tilde{\cA'}_{\epsilon,R, r}$).
\end{mydef}
We emphasize that $\epsilon$ is independent of $R$.

As such, each
  $f\in \mathcal{A}_{\epsilon, r}$
can be expressed as a convergent Taylor-Laurent series
$$
f(x,y,h,v)= \sum_{P\in \Z^{n-a-b}, Q \in \N^d}f_{Q,P}(x,y)h^Pv^Q
$$
for $(h,v)\in \Omega_{\epsilon_0,R} \times D(0, r(R))^d$ and $f_{Q,P} \in \cO(D(0,R)^a \times (D(0,R) \setminus D(0, \frac{1}{R}))^{b})$ for some $R>0$.

The following lemma is analogue of \cite{GS}[Lemma 4.13].
\begin{mylem}\label{dist-lemma} There is a constant $\kappa_0>0$ that depends only on $Im \gamma_1,\dots,Im \gamma_{n-a-b+q}$ such that if  $P \in\R^n$ and $\epsilon>\epsilon'$, there exists $Q \in {\mathcal P}_{\epsilon,R}$ such that for all $Q'  \in {\mathcal P}_{\epsilon',R}$ we have
$$
  (Q'-Q)\cdot P\leq  -\kappa_0(\epsilon-\epsilon')|P|.
$$
where
$${\mathcal P}_{\epsilon,R}=\{\sum_{n-a-b+1 \leq j \leq n+q-a-b} s_j Im \gamma_j+\sum_{ n+1+q-a-b\leq h \leq 2(n-a-b) } r_h Im \gamma_h \in \R^{n-a-b},$$$$  s_j \in ]-\epsilon,1+\epsilon[, r_h \in ]-R,R[  \}.$$
\end{mylem}
\begin{proof}
The proof follows \cite{GS}[Lemma 4.13]. Note that for any vector $Q  \in {\mathcal P}_{\epsilon,R}$ can be written as a sum of vector $Q_1$ in
$${\mathcal P}_{\epsilon}=\{\sum_{n-a-b+1 \leq j \leq n+q-a-b} s_j Im \gamma_j,  s_j \in ]-\epsilon,1+\epsilon[  \}
$$
and vector $Q_2$ in
$${\mathcal P'}_{R}=\left\{\sum_{n-a-b+1+q \leq j \leq 2( n-a-b)} r_h Im \gamma_h,  r_h \in ]-R,R[ \right \}.
$$
Apply \cite{GS}[Lemma 4.13] to $Q_1$ to get a vector $Q_1' \in {\mathcal P}_{\epsilon}$ such that
$$
  (Q'_1-Q_1)\cdot P\leq  -\kappa_0(\epsilon-\epsilon')|P|.
$$
Define $Q'=Q'_1+Q_2$ gives the conclusion.
\end{proof}
We have the following basic Cauchy estimates~:
\begin{mylem}
\label{cauchy}
 If $f=\sum_{P\in \Z^{n-a-b}}f_{P}(x,y,v)h^P=\sum_{Q\in \N^d}f_{Q}(x,y,h)v^Q \in \cA_{\epsilon , r}$   
 and  $0<\delta<{\kappa}\epsilon$, $\delta \leq \delta_0$,
	then for any $R >0$,
	\begin{equation}
\label{Reps}
	|f_{Q}(x,y,h)|\leq \frac{\sup_{\Omega_{\epsilon,R,r(R)}}|f|}{(r(R))^{|Q|}}.
\end{equation}
for $(x,y,h,v) \in \Omega_{\epsilon,R,r(R)}$.
	\begin{equation}
	\label{Reps2}
	\|f\|_{\epsilon
  {-{\delta}/{\kappa}}, R,r}\leq \frac{C\sup_{\Omega_{\epsilon,R, r}}|f|}{\delta^{\nu}},
	\end{equation}
		\begin{equation}
	\label{Reps3}
	\|\partial^{P_0}_h f\|_{\epsilon
  {-{\delta}/{\kappa}},R, r}\leq \frac{CC'^{|P_0|} \sup_{\Omega_{\epsilon,R,r}}|f|}{\delta^{\nu+|P_0|}},
	\end{equation}
	where $C$ and $\nu$ depends only on $n$ and $\delta_0$.
Here, $\kappa$ is some constant independent of $\epsilon,R$ and $C'$ depends only on $ \epsilon,R$.
\end{mylem}
\begin{proof}
The proof of the estimates $(\ref{Reps})$ is given in Lemma 4.15 of \cite{GS}.
To get the estimate $(\ref{Reps2})$  and $(\ref{Reps3})$,
we can modify the proof of Lemma 4.15 of \cite{GS} as follows.

According to \cite{GS}[Lemma 4.12] and Cauchy estimates for polydiscs, we have if $(x,y,h,v)\in \Omega_{\epsilon e^{-\delta/\kappa},R, r}$, then for all $s\in \Omega_{\epsilon,R}$ and any fixed $v$,
$$
|f_P(x,y,v)h^P|\leq \left|\frac{1}{(2\pi i)^{n-a-b}}\int_{|\zeta_1|=s_1,\dots, |\zeta_{n-a-b}|=s_{n-a-b}} f(x,y,\zeta,v)\frac{h^P}{\zeta^{P}}\, \frac{d\zeta_1\wedge\cdots \wedge d\zeta_{n-a-b}}{\zeta_1\cdots\zeta_{n-a-b}}\right|.
$$
Set $s_j=e^{-2\pi R_j}$, $|h_j|=e^{-2\pi R_j'}$, $R=(R_1,\cdots, R_{n-a-b})$ and $R'=(R_1',\cdots, R_{n-a-b}')$.  By Lemma \ref{dist-lemma},
		\begin{equation}
\label{fourier-type}
\inf_{(|\zeta_1|,\cdots,|\zeta_{n-a-b}|)=s\in\Omega_{\epsilon,R}}\sup_{h\in \Omega_{\epsilon  {-\delta},R}}
\left|\frac{h^P}{\zeta^P}\right|=\inf_{R\in \mathcal P_{\epsilon,R}}\sup_{R'\in\mathcal P_{\epsilon  {-\delta},R}} e^{-2\pi <R-R',P>}\leq e^{-\kappa\delta' |P|},
		\end{equation}
where the positive constant $\kappa$ depends only on $\Im \gamma_i$
and $\delta'=\delta/ \kappa$.  Thus
		\begin{equation}
\label{L^inf}
| f_{P}h^P|\leq \sup_{\Omega_{\epsilon,R,r(R)}}|f| e^{-\delta |P|}.
\end{equation}
Similarly, we have
$$
|\partial^{P_0}_h f(x,y,h,v)|\leq\sum_{P \in \Z^{n-a-b}}  \left|\binom{P}{P_0} f_{P}(x,y,v)h^{P-P_0}\right|
$$
where $P_0=(P_{0,1}, \cdots, P_{0,{n-a-b}})$ and
$$\binom{P}{P_0}= \prod_{j=1}^{n-a-b} \prod_{i=0}^{P_{0,j}-1} (P_j-i).$$
The estimate follows by summing and using (\ref{L^inf}) which gives
$$
|\partial^{P_0}_h f(x,y,h,v)|\leq C  \sup_{\Omega_{\epsilon,R,r(R)}}|f|  \prod_{i=1}^{n-a-b} s_i^{P_{0,i}} /\delta^{\nu+|P_0|}.
$$
\end{proof}
We will need the following type of Hartogs lemma.
\begin{mylem}\label{separate_analy}
 Let $F$ be a locally bounded function on $\mathrm{Conv} (\Omega) \times \Delta_t$ where $\Omega$ is a connected tube in $\C^n$ (i.e. there exists an open set $U \subset \R^n$ such that $\Omega=U \times \sqrt{-1} \R^n$), $\mathrm{Conv} (\Omega)$ its convex hull and $\Delta$ is a disc in $\C$.
 Assume that $F$ is holomorphic on $\Omega \times \Delta$.
 Assume that for any fixed $t_0 \in \Delta$, $F(\bullet, t_0)$ is holomorphic.
 Then $F$ is holomorphic on $\mathrm{Conv}(\Omega) \times \Delta$.
\end{mylem}
\begin{proof}
It is enough to prove the holomorphy near any point $(x,t) \in \mathrm{Conv}(\Omega) \times \Delta$.
Let $L$ be a 1 dimensional segment parallel to $\R^n$ such that $x \in L$ and $L \cap \Omega \neq \emptyset$.
Consider the tube
$S=L + \sqrt{-1} \R  T_x L$
which can be identified with a domain in $\C$.
Consider $(S \cap \Omega) \times \Delta$
which can be identified to be disjoint unions of $(]a_i,b_i[+\sqrt{-1}\R) \times \Delta$.
By \cite{Hor}[Lemma 2.2.11] and Archimedean property of $\R$,
$F$ is holomorphic on $S \times \Delta$.
In particular $\dbar_t F =0$.
\end{proof}

\begin{mylem}\label{convh}
Fix $\epsilon >0$ and a decreasing function $r: \R_{>0} \to \R_{>0}$.
Define
$$\Omega'_{\epsilon,r}:= \Omega_{\epsilon,r} \cup \cup_i (\hat\tau_i \Omega_{\epsilon,r} \cup \hat\tau_i^2 \Omega_{\epsilon,r}) \cup \cup_i( \hat\tau_i^{-1} \Omega_{\epsilon,r} \cup \hat\tau_i^{-2} \Omega_{\epsilon,r})$$
which is a Reinhardt domain.
Consider its
logarithmic indicatrix
$\omega'^*_{\epsilon,r}=\Omega'^*_\epsilon \cap (\C^{a+b} \times \R^{n-a-b} \times \C^d)$
with
$$\Omega'^*_{\epsilon,r}=\{\xi \in \C^{n+d}; (\xi_1, \cdots, \xi_{a+b}, e^{\xi_{a+b+1}}, \cdots, e^{\xi_n}, \xi_{n+1}, \cdots, \xi_{n+d}) \in \Omega'_{\epsilon,r}\}.$$
Denote
$$\varphi(z_1, \cdots , z_n, v_1, \cdots, v_d):=(z_1, \cdots, z_{a+b},\log| z_{a+b+1}|, \cdots ,\log| z_n|, v_1, \cdots, v_d),$$
$p(z_1, \cdots , z_n, v_1, \cdots, v_d)=(z_{a+b+1}, \cdots, z_n)$.
Let $F \in \cO(\Omega'_{\epsilon,r})$.
The $F$ can be extended over the preimage of the convex hull $\mathrm{Conv}(p(\omega'^*_{\epsilon,r}))$ of $p(\omega'^*_{\epsilon,r})$ under $\varphi \circ p$.
Similar statement holds for functions on
$$\Omega'_{\epsilon,R,r(R)}:= \Omega_{\epsilon,R,r(R)} \cup \cup_i (\hat\tau_i \Omega_{\epsilon,R,r(R)} \cup \hat\tau_i^2 \Omega_{\epsilon,R,r(R)}) \cup \cup_i( \hat\tau_i^{-1} \Omega_{\epsilon,R,r(R)} \cup \hat\tau_i^{-2} \Omega_{\epsilon,R,r(R)}).$$
Moreover, the $L^\infty$ norm of extended function from the domain $\Omega'_{\epsilon_0,R, r(R)}$ is equal to the $L^\infty$ norm of $F$ on $\Omega'_{\epsilon_0,R, r(R)}$.
\end{mylem}
\begin{proof}
By definition, it suffices to prove the result for any $\Omega'_{\epsilon,R,r(R)}$.
Consider the following mapping
$$\psi(z_1, \cdots , z_n, v_1, \cdots, v_d)=(z_1, \cdots, z_{a+b}, e^{z_{a+b+1}},  \cdots ,e^{ z_n}, v_1, \cdots, v_d).$$
We have for $z\in\C^n$, $v \in \C^d$, $\theta\in\R^n$
$$
\varphi\circ\psi(z,v)=(z_1, \cdots, z_{a+b}, \Re\, z_{a+b+1},\ldots, \Re\, z_n, v_1, \cdots, v_d).
$$
Hence, for fixed $z_1, \cdots, z_{a+b},v_1, \cdots, v_d$, the corresponding slice of $\Omega'^*_{\epsilon,R, r(R)}$ is a tube.
Let us consider the function $\psi^* F$. It is defined  for fixed $$(z_1, \cdots, z_{a+b},v_1, \cdots, v_d),$$ on the slice of $\Omega'^*_{\epsilon,R, r(R)}$ which is the tube over the slice of $\omega'^*_{\epsilon,R, r(R)}$. By a result of Bochner \cite{Bo38}, it extends over the preimage of the convex hull $\mathrm{Conv}(p(\omega'^*_{\epsilon,R, r(R)}))$  of $p(\omega'^*_{\epsilon,R, r(R)})$ under $\varphi \circ \psi$.
More precisely, the result of Bochner extends the function slice by slice.
Note that when restricting to each slice, the extended function is holomorphic.
Over each slice, $\psi^* |F|^2$ is subharmonic.
For any point in convex hull $\mathrm{Conv}(p(\omega'^*_{\epsilon,R, r(R)}))$  of $p(\omega'^*_{\epsilon,R, r(R)})$, by definition, there exist $x, y \in p(\omega'^*_{\epsilon,R, r(R)})$ such that the segment $L$ connecting $x, y$ is contained in $\mathrm{Conv}(p(\omega'^*_{\epsilon,R, r(R)}))$ containing the given point.
The restriction of $\psi^* |F|^2$ over the corresponding slice in $(p \circ \varphi \circ \psi)^{-1}(L)$ is still subharmonic.
In particular, by mean value inequality involving Poisson kernel (see e.g. \cite{agbook}[(4.12), Chap. I]),
its maximum over the corresponding slice in $(p \circ \varphi \circ \psi)^{-1}(L)$
is obtained on the corresponding slice in $(p \circ \varphi \circ \psi)^{-1}(\d L)$.
This implies
$$\sup_{(p \circ \varphi \circ \psi)^{-1}(\mathrm{Conv}(p(\omega'^*_{\epsilon,R, r(R)})))} \psi^* |F|^2=\sup_{( \varphi \circ \psi)^{-1}(\omega'^*_{\epsilon,R, r(R)})} \psi^* |F|^2.$$
In particular, $\psi^* F$ is bounded.
By Lemma \ref{separate_analy}, $\psi^* F$ is holomorphic.

On the other hand, for each variable, $\psi^* F$ is $2 \pi \sqrt{-1}-$periodic (in variables $z_{a+b+1}, \cdots, z_{n}$) over  the preimage of  $\omega'^*_\epsilon$ under $\varphi \circ \psi $ which is an open set in the connected preimage of its convex hull.
By identity theorem, the extension is unique and for each variable, $\psi^* F$ is $2 \pi \sqrt{-1}-$periodic (in variables $z_{a+b+1}, \cdots, z_{n}$)
which defines a holomorphic function on  the preimage of the convex hull $\mathrm{Conv}(p(\omega'^*_{\epsilon,r}))$ of $p(\omega'^*_{\epsilon,r})$ under $\varphi \circ p$.


\end{proof}
\begin{mylem}\label{cauchy-transl}
Under the same notations of previous lemma, there exists $\eta>0$ depending on $\epsilon$ such that
$$ \cup_i \hat\tau_i \Omega_{\epsilon+\eta,r} \cup \cup_i \hat\tau_i^{-1} \Omega_{\epsilon+\eta,r}$$
is contained in the preimage of the convex hull $\mathrm{Conv}(p(\omega'^*_{\epsilon,r}))$ of $p(\omega'^*_{\epsilon,r})$ under $\varphi \circ p$.
\end{mylem}
\begin{proof}
The statement is equivalent to $$p \circ \varphi(\cup_i \hat\tau_i \Omega_{\epsilon+\eta,r} \cup \cup_i \hat\tau_i^{-1} \Omega_{\epsilon+\eta,r}) \subset \mathrm{Conv}(\omega'^*_\epsilon)$$
for some $\eta >0$,
which is
invariant under the base change of $\R^{n-a-b}$.
In particular, without loss of generality, we may assume that $\hat\tau_i$ corresponds to translation by $e_j$ the standard basis of $\R^{n-a-b}$ and
$\omega'^*_{\epsilon,r}$ is the union of translation of the domain in
$$\C^a \times (\C^*)^b \times ]-\epsilon,1+\epsilon[^{q} \times \R^{n-a-b-q} \times \C^d$$
where it is easy to check the statement.
\end{proof}
\subsection{Neighborhoods of embedded toroidal manifolds}
Recalling notation (\ref{tauhat}), we will need the following sufficient condition to vertically/completely linearize the Deck transformations~:
\begin{mydef}\label{dioph}
	The pullback 
	normal bundle $\widetilde N_C$  is said to be  {strongly vertically Diophantine} if for all $(Q,P)\in \mathbb{N}^d\times \Z^{n-a-b}$,   $|Q|>1$ and all   $j=1,\ldots ,d$, $l\in \{1,\ldots, q\}$
	\begin{equation}
		\left |\lambda_l^P 
		\mu_{l}^Q-\mu_{{l},j}\right |   >  \frac{D}{(|P|+|Q|)^{\tau}}\label{dv}
	\end{equation}
	for some $D>0$, $\tau>0$ (independent of $P,Q$).
\end{mydef}
\begin{mydef}\label{dioph_full}
	The pullback 
	normal bundle $\widetilde N_C$  is said to be  {fully Diophantine} if for all $(Q,P)\in \mathbb{N}^d\times \Z^{n-a-b}$,   $|Q|>1$ and all $l\in \{1,\ldots, q\}$,  $j\in  \{1,\ldots ,d\}$, $i \in \{1,\ldots, n-a-b\}$
	\begin{equation}
		\max_{l \in \{1,\ldots, q\}}  \left |\lambda_l^P 
		\mu_{l}^Q-\lambda_{{l},i}\right |   >  \frac{D}{(|P|+|Q|)^{\tau}}\label{dh_full}
	\end{equation}
	\begin{equation}
		\max_{l \in \{1,\ldots, q\}}  \left |\lambda_l^P 
		\mu_{l}^Q-\mu_{{l},j}\right |   >  \frac{D}{(|P|+|Q|)^{\tau}}\label{dv_full}
	\end{equation}
	for some $D>0$, $\tau>0$ (independent of $P,Q$).
\end{mydef}
\begin{mydef}\label{def-cohom-op}
We define the (resp. ``inverse") vertical cohomological operator on $\tilde{\cA}_{\epsilon,r}^d$ (resp. $\tilde{\cA'}_{\epsilon,r}^d$)
\begin{equation}\label{linear-def}L_i^v(G):=G\circ\hat\tau_i-M_i G.
\end{equation}
(resp. \begin{equation}\label{linear-inv-def}L_{-i}^v(G):=G\circ\hat \tau_i^{-1}-M_i^{-1} G.)
\end{equation}
We define the (resp. ``inverse") horizontal cohomological operator on $\tilde{\cA}_{\epsilon,r}^{n-a-b}$ (resp. $\tilde{\cA'}_{\epsilon,r}^{n-a-b}$)
\begin{equation}\label{linear-def-horizontal}L_i^h(G):=G\circ\hat\tau_i-T_i G.
\end{equation}
(resp. \begin{equation}\label{linear-inv-def-horizontal}L_{-i}^h(G):=G\circ\hat \tau_i^{-1}-T_i^{-1} G.)
\end{equation}
Define the (resp. ``inverse") cohomological operator on $\tilde{\cA}_{\epsilon,r}^{n-a-b+d}$ (resp. $\tilde{\cA'}_{\epsilon,r}^{n-a-b+d}$)
\begin{equation}\label{linear-def-total}L_i:=(L_i^h,L_i^v).
\end{equation}
(resp. \begin{equation}\label{linear-inv-def-total}L_{-i}:=(L_{-i}^h, L_{-i}^v).)
\end{equation}
\end{mydef}
\begin{myprop}\label{cohomo}
	Assume   $N_C$
	is   
	strongly
	vertically
	Diophantine (resp. fully Diophantine).   Fix $\epsilon_0,r_0,\delta_0, \rho_0$  in $[0,1[$. Let $0<\epsilon<\epsilon_0$, $0<\rho<\rho_0$, $0<r<1$, $0<\delta<\delta_0$, and $\frac{\delta}{\kappa}<\epsilon$. Suppose that $F_i\in { \cA}_{\epsilon, r}^d$, $i=1,\ldots, q$,
	satisfy
	\begin{equation}
		\label{almost-com}
		L_i^\bullet(F_j^{\bullet})- L_j^\bullet(F_i^{\bullet})=0\quad
	\end{equation}
	(resp. 		\begin{equation}
		\label{almost-com2}
		L_{-i}^\bullet(F_j^{\bullet})- L_{-j}^\bullet(F_i^{\bullet})=0\quad
	\end{equation})
on  $\Omega_{\epsilon,R,r(R)} \cap \hat \tau_i^{-1}\Omega_{\epsilon,R,r(R)} \cap \hat \tau_j^{-1} \Omega_{\epsilon,R,r(R)}$ (resp. on  $\Omega_{\epsilon,R,r(R)} \cap \hat \tau_i\Omega_{\epsilon,R,r(R)} \cap\hat \tau_j \Omega_{\epsilon,R,r(R)}$) for any $R>0$.
Here $\bullet$ means "$v$" in the strongly vertically Diophantine case (resp. $\bullet$ means either "$h$" or "$v$" in the fully Diophantine case).
	There exist functions  $G \in  \tilde{ A}^d_{\epsilon
		{-{\delta}/{\kappa}},r e^{-\rho}}$
	(resp. $G' \in  \tilde \cA'^d_{\epsilon
		{-{\delta}/{\kappa}},re^{-\rho}}$)
	such that
	\begin{align}\label{formal2q+1}
		L_i^\bullet(G^{\bullet})&=F_i^{\bullet} \ \text{on   } \Omega_{\epsilon  {-{\delta}/{\kappa}},R,r(R) e^{-\rho}}.
	\end{align}
	(resp. 		 \begin{align}\label{formal2q+1}
		L_{-i}^\bullet((G')^{\bullet})&=F_i^{\bullet} \ \text{on   } \Omega_{\epsilon  {-{\delta}/{\kappa}},R,r(R) e^{-\rho}}.
	\end{align})
	Furthermore, $G$ satisfies
	\begin{align}\label{estim-sol}
		\|G\|_{\epsilon  {-{\delta}/{\kappa}},R,r(R)e^{-\rho}}&\leq  \max_i\|F_{i}\|_{\epsilon,R,r(R)}(\frac{C_1}{\delta^{\tau+\nu}}+\frac{C_1}{\rho^{\tau+\nu}}),\\
		\label{estim-solcompo}
		\|G\circ \hat \tau_i\|_{\epsilon  {-{\delta}/{\kappa}},R,r(R)e^{-\rho}}&\leq  \max_i\|F_{i}\|_{\epsilon,R,r(R)}(\frac{  C_1}{\delta^{\tau+\nu}}+\frac{C_1}{\rho^{\tau+\nu}}).
	\end{align}
	(resp.  $G'$ satisfies
	\begin{align}\label{estim-sol2}
		\|G'\|_{\epsilon  {-{\delta}/{\kappa}},R,r(R) e^{-\rho}}&\leq  \max_i\|F_{i}\|_{\epsilon,R,r(R)}(\frac{C_1}{\delta^{\tau+\nu}}+\frac{C_1}{\rho^{\tau+\nu}}),\\
		\label{estim-solcompo2}
		\|G'\circ \hat \tau_i^{-1}\|_{\epsilon  {-{\delta}/{\kappa}},R,r(R) e^{-\rho}}&\leq  \max_i\|F_{i}\|_{\epsilon,R,r(R)}(\frac{  C_1}{\delta^{\tau+\nu}}+\frac{C_1}{\rho^{\tau+\nu}}).
	\end{align})
	for some constant $\kappa, C_1$  that are independent of $F,\rho,\delta, r,\epsilon,R$ and $\nu$ that depends only on $n$ and $d$.

	Moreover, the solution $G$ (resp; $G'$) is unique.

\end{myprop}
\begin{proof}
	We only give the proof in the vertical cohomological operator case $(\ref{linear-def})$.
	The proof of another case is similar. We follows the proof of \cite{GS}[Proposition 4.17].
	First note that the formal solution of $G$ is unique and Lemma 1,
it is enough construct $G$ (resp. $G'$) for some fixed $R>0$ since every domain that we consider is Reinhardt. Since $F_i\in {\cA}_{\epsilon, r}$, we can write
	$$
	F_i(x,y,h,v)=\sum_{Q\in \mathbb{N}^d, |Q|\geq2}\sum_{P\in\Z^{n-a-b}}F_{i,Q,P}(x,y)h^Pv^Q,
	$$
	which converges normally   for $(x,y,h,v)\in\Om_{\epsilon,r}$.
	Note that $F_{i,Q,P}$
	are vectors, and its $k$th component
	is denoted  by $F_{i,k,Q,P}$ which is holomorphic in $(x,y) \in \C^a \times (\C^*)^{b}$.
	For each $(Q,P)\in \mathbb{N}^d\times \Z^{n-a-b}$,   each $i=1,\ldots, n-a-b$, and each $j=1,\ldots ,d$, let $ i_v:=i_v(Q,P,j)$ be in $ \{1,\ldots, q\}$ such that the maximum is realized in Definition \ref{dioph}.
	Let us set
	\begin{align}
		G_j&:=\sum_{Q\in \mathbb{N}^d, 2\leq |Q|}\sum_{P\in\Z^{n-a-b}}\frac{F_{i_v,j,Q,P}}{\lambda_{i_v}^P\lambda_{i_v}^Q-\mu_{i_v,j}}h^Pv^Q,\quad j=1,\ldots. d\label{sol-coh-v}.
	\end{align}
	According to (\ref{almost-com}), we have
	\begin{equation}\label{compat-coeff}
		(\lambda_{i_v}^P\mu_{i_v}^Q-\lambda_{i_v,i})F_{m,i,Q,P}=(\lambda_{m}^P\mu_{m}^Q-\lambda_{m,i}) F_{i_v,i,Q,P}.
	\end{equation}
	Therefore, using (\ref{compat-coeff}), the $i$th-component of $ L^v_m(G)$ reads
	\begin{align*}
		L^v_m(G_i) = &\sum_{Q\in \mathbb{N}^d, 2\leq|Q| }\sum_{P\in\Z^{n-a-b}}(\lambda_{m}^P\mu_{m}^Q-\lambda_{m,i})\frac{F_{i_v,i,Q,P}}{(\lambda_{i_v}^P\mu_{i_v}^Q-\lambda_{i_v,i})}h^Pv^Q\\
		= &\sum_{Q\in \mathbb{N}^d,2\leq|Q|}\sum_{P\in \Z^{n-a-b}}F_{m,i,Q,P}h^Pv^Q.
	\end{align*}
	Thus we have obtained, the formal equality~:
	\begin{equation}\label{formal-lin}
		L^v_m(G) = F_m 
		,\quad m=1,\ldots, n.
	\end{equation}
	Let us estimate these solutions.
	Without loss of generality, we may assume that $\tau \geq 1$.
	According to Definition \ref{dioph} and formula (\ref{sol-coh-v}), we have
	\begin{equation}
		\max_{j}\left (|G_{j,Q,P}|\right )\leq \max_i|F_{i,Q,P}|\frac{(|P|+|Q|)^{\tau}}{D}.
	\end{equation}
	Let $(h,v)\in \Omega_{\epsilon  {-{\delta}/{\kappa}},R, r(R)e^{-\rho}}$. According to (\ref{dv}), 
	we have by convexity
	\begin{align*}
		\|G_{Q,P}h^Pv^Q\|&\leq \max_i\|F_{i}\|_{\epsilon,R,r}e^{-\delta|P|- \rho |Q|}\frac{(|P|+|Q|)^{\tau}}{D}\\
		&\leq \max_i\|F_{i}\|_{\epsilon,R, r}e^{-\delta|P|- \rho |Q|}\frac{(|P|^\tau+|Q|^\tau)2^{\tau}}{D}\\
		&\leq \max_i\|F_{i}\|_{\epsilon,R, r}(e^{-\delta/2 |P|}\frac{(4\tau e)^{\tau}}{D\delta^{\tau}}+e^{- \rho/2 |Q|}\frac{(4\tau e)^{\tau}}{D\rho^{\tau}}).
	\end{align*}
	Summing over $P$ and $Q$, we obtain
	$$
	\|G\|_{\epsilon  {-{\delta}/{\kappa}},R,r(R)e^{-\rho}}\leq \max_i\|F_{i}\|_{\epsilon,R, r}(\frac{C'}{\delta^{\tau+\nu}}+\frac{C'}{\rho^{\tau+\nu}}),
	$$
	for some constants $C',\nu$ that are independent of $F,\epsilon,\delta, \rho,R$. Hence, $G\in
	\cA^d_{\epsilon  {-{\delta}/{\kappa}},R,r(R)e^{-\rho}}$. 

	Let us prove (\ref{estim-solcompo}).   Let $B:=2\max_{\ell,j}|\mu_{\ell,j}|$. Then, there is a constant $D'$ such that
	\begin{equation}\label{sd-ehanced}
		\max_{\ell\in \{1,\ldots, q\}}  \left |\lambda_{\ell}^P\mu_{\ell}^Q-\mu_{{\ell},j}\right | \geq  \frac{D'\max_k|\lambda_{k}^P\mu_{k}^Q|}{(|P|+[Q|)^{\tau}}.
	\end{equation}
	Indeed, if $\max_k|\lambda_k^P\mu_k^Q|<B$, then Definition \ref{dioph} gives (\ref{sd-ehanced}) with $D':= \frac{D}{B}$. Otherwise, if $$|\lambda_{k_0}^P\mu_{k_0}^Q|:=\max_k|\lambda_{k}^P\mu_{k}^Q|\geq B,$$
	then $|\mu_{k_0,i}|\leq \frac{B}{2}\leq \frac{|\lambda_{k_0}^P\mu_{k_0}^Q|}{2}$. Hence, we have
	$$
	\left |\lambda_{k_0}^P\mu_{k_0}^Q-\mu_{k_0,i}\right|\geq \left||\lambda_{k_0}^P\mu_{k_0}^Q|- |\mu_{k_0,i}|\right|\geq \frac{|\mu_{k_0}^P\mu_{k_0}^Q|}{2}.
	$$
	We have verified (\ref{sd-ehanced}).
	Finally, combining all cases gives us, for $m=1,\ldots, q$,
	\begin{align*}
		|[G\circ\hat \tau_m]_{QP}| &= \left| G_{Q,P}\lambda_{m}^P\mu_{m}^Q
		\right|\leq\max_{\ell}|F_{\ell,Q,P}| \frac{|\lambda_{m}^P\mu_{m}^Q|}{|\lambda_{i_v}^P\mu_{i_v}^Q-\mu_{i_v,i}|}\\
		&\leq \max_{\ell}|F_{\ell,Q,P}|\frac{|\lambda_{m}^P\mu_{m}^Q|(|P|+|Q|)^{\tau}}{D'\max_k|\lambda_{k}^P\mu_{k}^Q|}\\
		&\leq  \max_{\ell}|F_{\ell,Q,P}|\frac{(|P|+|Q|)^{\tau}}{D'}.\end{align*}
	Hence,  $\tilde G_m:=G\circ \hat \tau_m\in \cA^d_{\epsilon  {-{\delta}/{\kappa}},R,r(R)e^{-\rho}}$.   We can define  $\tilde G\in \tilde{\cA}^d_{\epsilon  {-{\delta}/{\kappa}},R,r(R)e^{-\rho}}$ such that $\tilde G=\tilde G_m \circ \hat \tau_m^{-1}$ on $\hat \tau_m(\Omega_{\epsilon,R,r(R)})$.  We verify that $\tilde G$ extends to a single-valued holomorphic function of class $\tilde {\cA}^d_{\epsilon,R,r}$. Indeed,
	$\tilde G_i\circ \hat \tau_i^{-1}=\tilde G_j\circ \hat \tau_j^{-1}$ on $\hat \tau_i(\Omega_{\epsilon,R,r(R)})\cap \hat \tau_j(\Omega_{\epsilon,R,r(R)})$, since the latter is connected  and the two functions agree with $G$ on $\hat \tau_i(\Omega_{\epsilon,R,r(R)})\cap \hat \tau_j(\Omega_{\epsilon,R,r(R)})\cap \Omega_{\epsilon,R,r(R)}$ that contains a neighborhood of $\C^a \times (\C^*)^{b} \times \Omega_\epsilon \times\{0\}$ in $\C^{n+d}$.

	The uniqueness follows from the uniqueness as formal solution.
\end{proof}
\begin{myprop}\label{str-cohomo}
Assume   $N_C$ is vertically strongly Diophantine. For each $1\leq i\leq q$, let
$$
F_i(h,v)=\sum_{Q\in \mathbb{N}^d, |Q|\geq2}\sum_{P\in\Z^{n-a-b}}F_{i,Q,P}h^Pv^Q,
$$
be q-dimensional vector of formal power series in $v$, with Laurent series in $h$ as coefficients. We assume that they satisfy the formal relations, $1\leq i,j\leq q$, $L_i^v(F_j)=L_j^v(F_i)$, that is $\forall (k,Q,P)\in\{1,\ldots,q\}\times \N^d \times\Z^{n-	a-b}$
$$
(\lambda_{i}^P\mu_{i}^Q-\lambda_{i,k})F_{j,k,Q,P}=(\lambda_{j}^P\mu_{j}^Q-\lambda_{j,k}) F_{i,k,Q,P}.
$$
Then, for each $1\leq i\leq q$, there exists a unique formal power series
\begin{equation}\label{formal'2q+1}
	G^{(i)}:= \left(\sum_{Q\in \mathbb{N}^d, |Q|\geq2}\sum_{P\in\Z^{n-a-b}}G_{k,Q,P}^{(i)}h^Pv^Q\right)_{1\leq k\leq q}\text{such that}\quad	L_i^v(G^{(i)})= F_i.
\end{equation}
Furthermore, if, for a given $1\leq i\leq q$, $F_i$ is holomorphic on  $ \hat \tau_i^{-2}(\Omega_{\epsilon,R,r})$ (resp. on  $ \hat \tau_i^{2}\Omega_{\epsilon,R,r} $),  then $G^{(i)}$ is holomorphic on  $\hat \tau_i^{-2}(\Omega_{\epsilon-\frac{\delta}{\kappa},R,re^{-\rho}})$ for any $0<\delta<\kappa\epsilon$ and $0<\rho$ and  satisfies
\begin{align}\label{estim'-sol}
	\|G^{(i)}\circ \hat \tau^{-2}_i \|_{\epsilon  {-{\delta}/{\kappa}},R,re^{-\rho}}&\leq \|F_{i}\|_{\hat \tau_i^{-2}(\Omega_{\epsilon,R,r})}(\frac{C_1}{\delta^{\tau+\nu}}+\frac{C_1}{\rho^{\tau+\nu}}),\\
	\label{estim'-solcompo}
	\|G^{(i)}\circ \hat \tau^{-1}_i\|_{\epsilon  {-{\delta}/{\kappa}},R,re^{-\rho}}&\leq \|F_{i}\|_{\hat \tau_i^{-2}(\Omega_{\epsilon,R,r})}(\frac{  C_1}{\delta^{\tau+\nu}}+\frac{C_1}{\rho^{\tau+\nu}}).
\end{align}
for some constant $\kappa, C_1$  that are independent of $\{F_i\}_i,\rho,\delta, R,\epsilon$ and $\nu$ that depends only on $n$ and $d$. Replacing $L_i^v$ by $L_{-i}^v$ yields, for each $1\leq i\leq q$, a unique formal power series $G^{(-i)}$ satisfying  $L_{-i}^v(G^{(-i)})=F_i$ with estimates
\begin{equation}\label{estim'-sol2}
	\max_{k=1,2}\left(\|G^{(-i)}\circ \hat \tau_i^k\|_{\epsilon  {-{\delta}/{\kappa}},R,re^{-\rho}}\right)\leq (\frac{  C_1}{\delta^{\tau+\nu}}+\frac{C_1}{\rho^{\tau+\nu}})\|F_{i}\|_{\hat \tau_i^{2}(\Omega_{\epsilon,R,r})},
\end{equation}
as above if  $F_i$ is holomorphic on  $ \hat \tau_i^{2}(\Omega_{\epsilon,R,r})$.
\end{myprop}
\begin{proof}
	Indeed, for each $1\leq i\leq q$, the formal solution $G^{(i)}$ for the $i$th vertical cohomological equation (\ref{formal'2q+1}) is given by~:
	\begin{align}
		G_j^{(i)}&:=\sum_{Q\in \mathbb{N}^d, 2\leq |Q|}\sum_{P\in\Z}\frac{F_{i,j,Q,P}}{\lambda_{i}^P\mu_{i}^Q-\mu_{i,j}}h^Pv^Q,\quad j=1,\ldots. d\label{sol'-coh-v}.
	\end{align}
	We recall that under vertically strongly  Diophantine condition, none of the denominator vanishes.
	Thus the above formal solution is meaningful.
	Moreover, 
	Diophantine inequality yields
	\begin{equation}
		|G_{i,Q,P}|\leq |F_{i,Q,P}|\frac{(|P|+|Q|)^{\tau}}{D}.
	\end{equation}
	The rest of the proof is identical to that of Proposition \ref{cohomo}.
	\end{proof}
\begin{myrem}\label{i-domain}
	Assume that $ N_C$ is vertically strongly Diophantine.
	Let $F_i\in { \cA}^d_{\epsilon, r}$, $i=1,\ldots, q$.
	Let us express $F_i$ as formal power series
	$$
	F_i(h,v)=\sum_{Q\in \mathbb{N}^d, |Q|\geq2}\sum_{P\in\Z^{n-a-b}}F_{i,Q,P}h^Pv^Q
	$$
	normally convergent on $\Omega_{ \epsilon,r}$ satisfying to $L^v_i(F_j)=L^v_j(F_i)$ for all $1\leq i,j\leq q$.
	Assume furthermore that, for each $1\leq i\leq q$, $F_i$ is also holomorphic in a neighborhood of $\hat\tau_i^{-2}(\Om_{\epsilon,r})\cup \hat\tau_i^{-1}(\Om_{\epsilon,r})$. Let us set for $\ell\in\N$ and $1\leq i\leq q$~:
	 \begin{align}
	 	\tilde\Om^{(\pm \ell)}_{i,\epsilon,r}&:=\Om_{\epsilon,r}\cup\hat\tau_i^{\pm 1}(\Om_{\epsilon,r})\cup \cdots\cup \hat\tau_i^{\pm \ell}(\Om_{\epsilon,r})\nonumber\\
	 	\tilde\Om^{(\pm \ell)}_{\epsilon,r}&=\cup_{i=1}^{q}\tilde\Om^{(\pm \ell)}_{i,\epsilon,r}\label{domainsl}.
	 \end{align}
	Hence, for each $1\leq i\leq q$, $F_i$ is holomorphic in a neighborhood of $\tilde\Om^{(-2)}_{i,\epsilon,r}$.
	On the one hand, according to Proposition \ref{cohomo}, there exist a unique solution $G$ holomorphic on $\tilde\Om_{\epsilon-\frac{\delta}{\kappa},re^{-\rho}}^{(-1)}\cup\tilde\Om_{\epsilon-\frac{\delta}{\kappa},re^{-\rho}}^{(1)} $ satisfying to $L^v_i(G)=F_i$ for all $1\leq i\leq q$. On the other hand, according to Proposition \ref{str-cohomo}, for each $1\leq i\leq q$, there exists a unique solution $G^{(i)}$ holomorphic on neighborhood of $\hat\tau_i^{-2}(\Om_{\epsilon-\frac{\delta}{\kappa},re^{-\rho}})\cup \hat\tau_i^{-1}(\Om_{\epsilon-\frac{\delta}{\kappa},re^{-\rho}})$ to equation $L^v_i(G^{(i)})=F_i$. Since, for each $i$, $\hat\tau_i^{-2}(\Om_{\epsilon-\frac{\delta}{\kappa},re^{-\rho}})\cup \hat\tau_i^{-1}(\Om_{\epsilon-\frac{\delta}{\kappa},re^{-\rho}})$ intersects $\tilde\Om^{(-1)}_{\epsilon-\frac{\delta}{\kappa},re^{-\rho}}\cup \tilde\Om^{(1)}_{\epsilon-\frac{\delta}{\kappa},re^{-\rho}}$ along a single connected component, $G^{(i)}$ is the holomorphic extension of $G$ on $\hat\tau_i^{-2}(\Om_{\epsilon-\frac{\delta}{\kappa},re^{-\rho}})\cup \hat\tau_i^{-1}(\Om_{\epsilon-\frac{\delta}{\kappa},re^{-\rho}})$. Hence, $G$ is a holomorphic function on neighborhood of $\tilde\Om^{(-2)}_{\epsilon-\frac{\delta}{\kappa},re^{-\rho}}$ with estimates, for each $1\leq i\leq q$, $R$
	\begin{equation}\label{estim+i-20}
		\sup_{\tilde\Om^{(-1)}_{\epsilon-\frac{\delta}{\kappa},R, re^{-\rho}}\cup\tilde\Om^{(1)}_{\epsilon-\frac{\delta}{\kappa},R,re^{-\rho}}}\|G(h,v)\|\leq \sup_{\Om_{\epsilon,R,r}}\|F_i\|(\frac{  C_1}{\delta^{\tau+\nu}}+\frac{C_1}{\rho^{\tau+\nu}}).
	\end{equation}
	\begin{equation}\label{estim+i-2}
		\sup_{\hat\tau_i^{-2}(\Om_{\epsilon-\frac{\delta}{\kappa},R,re^{-\rho}})\cup \hat\tau_i^{-1}(\Om_{\epsilon-\frac{\delta}{\kappa},R,re^{-\rho}})}\|G\|\leq \sup_{ \hat\tau_i^{-2}(\Om_{\epsilon,R,r})}\|F_i\|(\frac{  C_1}{\delta^{\tau+\nu}}+\frac{C_1}{\rho^{\tau+\nu}}).
	\end{equation}
	Similarly, assume that, for each $1\leq i\leq q$, $\tilde F_i$ is holomorphic in a neighborhood of $\tilde\Om^{(2)}_{i,\epsilon,r}$ and satisfying to $L^v_{-i}(\tilde F_j)= L^v_{-j}(\tilde F_i)$ for all $i,j$. Then there exists a unique $\tilde G$ holomorphic on neighborhood of $\tilde\Om^{(2)}_{\epsilon-\frac{\delta}{\kappa},re^{-\rho}}$ satisfying to $L^v_{-i}(\tilde G)=\tilde F_i$ with estimates similar to the above ones and written compactly as~: for each $1\leq i\leq q$
	\begin{equation}\label{estim-i+2}
		\sup_{(h,v)\in\tilde\Om^{(2)}_{i,\epsilon-\frac{\delta}{\kappa},R, re^{-\rho}}}\|\tilde G(h,v)\|\leq \sup_{\Om_{\epsilon,R,r}\cup \hat\tau_i^{2}(\Om_{\epsilon,R,r})}\|\tilde F_i\|(\frac{  C_1}{\delta^{\tau+\nu}}+\frac{C_1}{\rho^{\tau+\nu}}).
	\end{equation}
	\end{myrem}
\subsection{Ueda problem for toroidal manifolds}

\begin{mythm}
\label{ueda-thm}
	Let $C$ be an $n$-dimensional complex abelian group, holomorphically embedded into a complex manifold $M_{n+d}$.   Assume that $T_{M|_C}$ splits. Assume the normal bundle $N_C$ has (locally constant) unitary transition functions. Assume that  $N_C$ is vertically  Diophantine (see Definition (\ref{dioph})). Then there exists a non-singular
	holomorphic foliation of the germ of neighborhood of $C$ in $M$ having $C$
	as a leaf.
\end{mythm}
We are interested in the existence of a non-singular
holomorphic foliation of the germ of neighborhood of $C$ in $M$ having $C$
as a leaf. We refer to it as a``horizontal foliation” if exists.
The above-mentioned ``horizontal foliation” will be obtained if
we can find $\Phi= \mathrm{Id}+\phi$ be a biholomorphism of  $\Omega_{\epsilon_0,r}$ (to be chosen) such that for any $1 \leq i \leq q$,
\begin{equation}
\label{ueda-equ}
\Phi \circ \tilde \tau_i =\tau_i \circ \Phi 
\end{equation}
for some biholomorphism of $\Omega_{\epsilon_0,r}$ (to be chosen with $r$ decreasing)
$$ \tilde \tau_i(x,y,h,v)=( \tilde \tau_i^h(x,y,h,v), M_iv)$$
such that $(M,C)$ is biholomorphic to the quotient of $\Omega_{\epsilon_0,r}$ by  $\tilde\tau_1,\dots,\tilde \tau_q$.
Indeed, the codimension $d$ ``horizontal foliation” can be defined locally by
$dv=0$ in each coordinate chart which glue to be a global foliation since the transition functions are constants.
Note that $C$ is the leaf defined locally by $\{v=0\}$.

By the identity theorem, the equation $(\ref{ueda-equ})$ implies that
\begin{equation}
\Phi \circ \tilde \tau_i^{-1} =\tau_i^{-1} \circ \Phi \label{Phi_inv}
\end{equation}
whenever both sides are well defined. Such $ \tilde \tau_i$ is called a vertical linearization of $(M,C)$.

We want to find $\phi$ which is of the form $\phi(x,y,h,v)=(0, \phi^v(x,y,h,v))$ .
Denote
$$\tau_i^{\pm 1}=(\tau^h_{i,\pm}, \tau^v_{i,\pm}).$$
To simplify the notations, we will sometimes omit the positive sign.
Assume that they are defined on $ \Omega_{\epsilon_0,r}$
for suitable choice (sufficiently small) $\epsilon_0$ (where $r$ is a function of $R$) such that $(C,M)$ is biholomorphic to the quotient of $ \Omega_{\epsilon_0,r}$ by $\tau_j(1 \leq j \leq q)$.
Assume also that they are defined on
$$ \cup_{1 \leq i \leq q, \pm} \hat\tau_i^{\pm 1}(\Omega_{\epsilon_0, r}) \cup \cup_{1 \leq i \leq q, \pm} \hat\tau_i^{\pm 2}(\Omega_{\epsilon_0, r})$$
for the same choice $\epsilon_0, r$.

Applying Lemma \ref{convh} with  $\epsilon_0$, there exists $\eta>0$ depending on $\epsilon_0$ such that
$$ \cup_i \hat \tau_i \Omega_{\frac{\epsilon_0}{2}+\eta,r} \cup \cup_i (\hat \tau_i)^{-1} \Omega_{\frac{\epsilon_0}{2}+\eta,r}$$
is contained in the preimage of the convex hull $\mathrm{Conv}(p(\omega'^*_{\epsilon,r}))$ of $p(\omega'^*_{\epsilon,r})$ under $\varphi \circ p$ with the same notations of Lemma \ref{convh}.

Define the higher order perturbations
$$\tilde \tau_{i, \pm}^{*,h}:=\tilde \tau_{i,\pm}^{h}-(\mathrm{id},\mathrm{id},  T_i^{\pm 1}), \tau_{i,\pm}^{*,h}:= \tau_{i, \pm}^{h}-(\mathrm{id},\mathrm{id},T_i^{\pm 1}),$$
$$\tau_{i, \pm}^{*,v}:= \tau_{i, \pm}^{v}-M_i^{\pm 1}.$$
We use the abuse of notation of denoting a linear map defined by a matrix by the same symbol.
The horizontal part of equation $(\ref{ueda-equ})$ is given by
\begin{equation}\label{horizontal}
\tilde \tau_{i, \pm}^{*,h}(x,y,h,v)=\tau_{i, \pm}^{*,h}(x,y,h,v+\phi^{v}(x,y,h,v)).
\end{equation}
The vertical part of equation $(\ref{ueda-equ})$ is given by
\begin{equation}\label{vertical}
\phi^v((x,y,T_i^{\pm 1} h)+\tilde \tau_{i, \pm}^{*,h}(x,y,h,v),M_i^{\pm 1} v) =M_i^{\pm 1} \phi^v(x,y,h,v)+\tau_{i, \pm}^{*,v}(x,y,h,v+\phi^{v}(x,y, h,v)).
\end{equation}
We recall the vertical cohomological operator~: $$L_i^v(\phi^v):=\phi^v(x,y,T_i h, M_i v)-M_i \phi^v(x,y,h,v).$$
Using equations $(\ref{horizontal})$ and $(\ref{vertical})$, we have
\begin{align}\label{linear+}
L_i^v(\phi^v)(x,y,h,v)&=\tau_{i, +}^{*,v}(x,y,h,v+\phi^v(x,y,h,v))-(\phi^v((x,y,T_ih)+\nonumber\\ &\tau_{i,+}^{*,h}(x,y,h, v+\phi^v(x,y,h,v)), M_i v)-\phi^v(x,y,T_ih, M_i v)).
\end{align}
This is defined by developing the horizontal and vertical parts of equation (\ref{Phi_inv}).
Using equations $(\ref{horizontal})$ and $(\ref{vertical})$, we have
\begin{align}\label{linear-}
L_{-i}^v(\phi^v)(x,y,h,v)&=\tau_{i, -}^{*,v}(x,y,h,v+\phi^v(x,y,h,v))-(\phi^v((x,y,T_i^{-1}h)+\nonumber\\ &\tau_{i,-}^{*,h}(x,y,h, v+\phi^v(x,y,h,v)), M_i^{-1} v)-\phi^v(x,y,T_i^{-1}h, M_i^{-1} v)).
\end{align}
In the following, we will estimate the $L^\infty$ norm to show that this formal solution is in fact convergent.
To do so, we will follow the majorant method in \cite[Section 3.3]{GS21}.
Denote
\begin{equation}\label{(I)}
{ (I)_{\pm}:=}\tau_{i, \pm}^{*,v}(x,y,h,v+\phi^v(h,v)) ;
\end{equation}
\begin{equation}\label{(II)}
(II)_{\pm}:=\phi^v((x,y,T_i^{\pm 1}h)+\tau_{i, \pm}^{*,h}(x,y,h, v+\phi^v), M_i^{\pm 1} v)-\phi^v(x,y,T_i^{\pm 1}h, M_i^{\pm 1} v).
\end{equation}
The major difference compared to \cite[Section 3.3]{GS21} will be the estimates for $(II)$.

In the following, we will estimate $[\phi^v]_k(k \geq 2)$ by induction on $k$ (which gives the estimate for $\phi^v$.)
By identity theorem, we get the same $\phi^v$ either by the vertical cohomological operator or the inverse vertical cohomological operator if the solutions are holomorphic.

Let $0<r_1<\min(1,r)$ and $0<\epsilon_1<\epsilon_0$ be positive constants to be chosen below sufficiently small. Let us define the sequences  $r_{m+1}=r_me^{-\frac{1}{2^{m}}}$ and $\epsilon_{m+1}=\epsilon_m- \epsilon_1\frac{\eta}{ 2^{m}\kappa}$ for $m>0$ and some $\eta<\frac{\kappa}{2}$ sufficiently small. We have  $r_{m+1}:=r_1e^{-\sum_{k=1}^m\frac{1}{2^k}}$ and $\epsilon_{m+1}:= \epsilon_1-\epsilon_1\sum_{k=1}^m\frac{\eta}{2^k  \kappa}$ for $m\geq 1$. We have, for $m\geq 1$
	\begin{equation}\label{minorants}
		r_m>r_1e^{-1},\quad  \epsilon_{m}>\frac{\epsilon_1}{2}.
	\end{equation}

Our goal to find germs of holomorphic function at 0
$$A(t)=\sum_{k \geq 2} A_k t^k,$$
and for $1 \leq i \leq q$,
$$B^{\pm e_i}(t)=\sum_{k \geq 2} B^{\pm e_i}_k t^k, 1 \leq i \leq q,$$
such that
for any fixed $R>0$, any $x,y$,
\begin{equation}\label{goal}
\sup_{(h,v) \in\cup_{i=0}^q \hat\tau_i^{\pm 1}(\Omega_{\epsilon_k,R, r_k}) }| [\phi^v]_k(x,y,h,v)|  \leq A_k \eta_k,
\end{equation}
\begin{equation}\label{goal'}
\sup_{(h,v) \in\hat\tau_i^{\pm 1}(\Omega_{\epsilon_k,R, r_k}) \cup \hat\tau_i^{\pm 2}( \Omega_{\epsilon_k,R, r_k})} | [\phi^v]_k(x,y,h,v)| \leq B^{\pm e_i}_k \eta_k,
\end{equation}
for suitable chosen $ r_1$.
Here $A_k, \eta_k$ depends on $R$.
 Here, the sequence $\{\eta_m\}_{m\geq 1}$ is defined by $\eta_1=1$ and for $m\geq 2$
\begin{equation}
	\eta_m
	:=\frac{C_1}{\eta^{\tau+\nu}} 2^{m(\tau+\nu)}
	\max_{m_1+\cdots +m_p+s=m} \eta_{m_1}\cdots \eta_{m_p}.\label{def-eta}
	\end{equation}
where the constant $C_1$ is defined in Proposition \ref{cohomo}. Here we have $1\leq m_i\leq m-1$ and $s\in \N$.
 We have
\begin{equation}
\eta_m\leq\max_{1\leq s\leq m-1}\left(\frac{C_1}{\eta^{\tau+\nu}}\right)^{m-s+1}  2^{(\tau+\mu)(2m-s)}
\leq D^m,
\end{equation}
for some positive constant $D$.

Define as formal series
$$J^{m-1}A(t):=A_2t^2+\cdots+ A_{m-1}t^{m-1},$$
$$A(t)=\sum_{m \geq 2} A_m t^m,\quad B^{\pm e_i}(t)=\sum_{m \geq 2} B^{\pm e_i}_m t^m.
$$
$$$$
Consider the Taylor development
$$\phi^v(x,y,h,v)=\sum_{Q \in \N^d, |Q| \geq 2} \phi_Q(x,y,h)v^Q.$$
Let $[\phi^v]_k$ be the homogeneous degree $k$ part of $\phi^v$
$$[\phi^v]_k(x,y,h,v)=\sum_{Q \in \N^d, |Q| =k} \phi_Q(x,y,h)v^Q.$$
In  the following, we will always denote $[\bullet]_k$ to indicate the homogeneous degree $k$ part of some series in $v$.
Notice that
$$[L_{\pm i}^v(\phi^v)]_k=L_{\pm i}^v [\phi^v]_k.$$
The degree 2 part  is
$$L_i^v([\phi^v]_2)= [\tau_{i,+}^{*,v}(x,y,h,v)]_2,$$
whose right-handed-side term is independent of $\phi^v$.
Similarly, the degree 2 part for the inverse vertical cohomological operator is
$$L_{-i}^v([\phi^v]_2)= -[\tau_{i,+}^{*,v}(x,y,h,v)]_2,$$
whose right-handed-side term is independent of $\phi^v$. According to the following equation
\begin{equation}\label{compat}
	L_i^v([\tau_{j}^{*,v}]_{m})= L_j^v([\tau_{i}^{*,v}]_{m})
	\end{equation}
and Proposition \ref{cohomo}, these two sets of equations have the same unique solution $[\phi^v]_2$ on $\Omega_{ \epsilon_1,r_1}$.

We proceed by induction on $m\geq 2$ as
	Taylor expansion at $v=0$ of (\ref{linear+}) shows that for any $m \geq 2$,
	\begin{equation}\label{degm}L_i^v [\phi^v]_m=P_i(x,y,h; v, [\phi^v]_2, \cdots, [\phi^v]_{m-1})\end{equation}
	where $P_i(x,y,h; v, [\phi^v]_2, \cdots, [\phi^v]_{m-1})$ is analytic in $x,y,h$ and polynomial in $v$, $[\phi^v]_2$, ..., $[\phi^v]_{m-1}$.
To obtain the estimate (\ref{goal}) of homogeneous part of degree $m$, $[\phi^v]_m$, we invert and estimate the solution of the vertical cohomological operator $L_i^v$ from equation (\ref{degm}). To do so, it is sufficient by Proposition \ref{cohomo},
to estimate the norm  of the homogeneous part of degree $m$ of its right hand side $(I)+(II)$, on $\Omega_{\epsilon_{m-1},R, r_{m-1}}$.

According to Cauchy estimates, the $L^{\infty}$-estimate of  term $(II)$ needs estimate of terms of degree $m_1 \leq m-1$ on
$\hat\tau_i(\Omega_{\epsilon_{m_1}+ \eta,R, r_{m_1}})$
(resp. $\hat\tau_i^{-1}(\Omega_{\epsilon_{m_1}+ \eta,R, r_{m_1}})$).
We notice that, according to Cauchy estimates, the $L^{\infty}$-estimate of  term $(II)_{i,\pm}$ needs estimate of terms of degree $m_1 \leq m-1$ on
$\hat\tau_i(\Omega_{\epsilon_{m_1}+ \eta,R, r_{m_1}})$
(resp. $\hat\tau_i^{-1}(\Omega_{\epsilon_{m_1}+ \eta,R,r_{m_1}})$). The latter is obtained by induction.

We will focus on the  vertical cohomological equation (\ref{linear+}).
The case of the inverse  vertical cohomological equation (\ref{linear-}) is obtained similarly. We omit the "+" index in the following.
%

Let us estimate the norms of (I) and (II).

Denote $\N^d_k:=\{Q\in\N^d\colon|Q|\geq k\}$. Let $m\geq 2$, for $Q\in\mathbb N^d_2$, $|Q|\leq m$, let us set
$$
E_{Q,m}=\left\{(m_{1,1},\ldots,m_{1,q_1},\ldots,m_{d,1},\ldots,m_{d,q_d})\in\mathbb N^{|Q|}_1\colon \sum_{i=1}^d{m_{i,1}+\cdots+m_{i,q_i}=m}\right\}.
$$
For $Q\in \N^d_2$, we have
$$
\left[(v+\phi^v(h,v))^Q\right]_m= \sum_{
M\in E_{Q,m}}\prod_{j=1}^d[v_j+\phi^v_{j}]_{m_{j,1}}\cdots [v_j+\phi^v_{j}]_{m_{j,q_j}}
$$
where we have set $\phi^v=(\phi^v_{1},\ldots, \phi^v_{d})$ and $M=(m_{1,1},\ldots,m_{1,q_1},\ldots,m_{d,1},\ldots,m_{d,q_d})$.
Thus we have for any $\epsilon>0$ and $r>0$ for which $[\phi^v]_l$ is well defined $l< m$,
\begin{equation}
	\left\|\left[(v+\phi^v(h,v))^Q\right]_m\right\|_{\epsilon, R,r}\leq \sum_{
	M\in E_{Q,m}}\prod_{j=1}^d \left\|[v_j+\phi^v_{j}]_{m_{j,1}}\right\|_{\epsilon,R,r}\cdots \left\|[v_j+\phi^v_{j}]_{m_{j,q_j}}\right\|_{\epsilon,R,r}.
\end{equation}

Let $M'_i=(m_{1,1}^{(i)},\ldots,m_{1,q_1^{(i)}}^{(i)},\ldots,m_{d,1}^{(i)},\ldots,m_{d,q_d^{(i)}}^{(i)})\in\mathbb N^{|Q^{(i)}|}_1$
with $|Q^{(i)}|\leq m_i$
and $
m_i=\sum_{j=1}^d m_{j,1}^{(i)}+\cdots+m_{j,q_j^{(i)}}^{(i)}$, $i=1,2$.  Define the concatenation
$M'_1\sqcup M'_2$ 
to be  $(M'_1,M'_2)$.
Hence, we emphasize that the concatenation
\begin{equation}\label{concatenation}
	\left(\bigcup_{2\leq |Q_1|\leq m_1}E_{Q_1,m_1}\right)\sqcup \left(\bigcup_{2\leq |Q_2|\leq m_2}E_{Q_2,m_2}\right)\subset \bigcup_{2\leq |Q|\leq m_1+m_2}E_{Q,m_1+m_2}.
\end{equation}
By Cauchy estimate $(\ref{Reps})$ applying to $\tau_{i,j}(1 \leq j \leq n-a-b+d)$ implies that, if $\epsilon_1$ small enough, there exists $R'>0$ (depending on $R$) 
such that
$$
||\tau_{i,Q,j}||_{\epsilon_1,R,r} \leq R'^{|Q|}
$$
with
$$\tau_{i,j}=\sum_{Q \in \N^d} \tau_{i,Q,j} (x,y,h)v^Q$$
and $\tau^v_i=(\tau_{i,1}, \cdots, \tau_{i,d})$.
Without loss of generality, we may assume that same estimate holds on
$$
\tilde \Omega_{\epsilon_1,R,r}^{(-2)}\cup\tilde \Omega_{\epsilon_1,R,r}^{(2)}:=\cup_{j} \hat\tau_j^{-2}\Omega_{\epsilon_1,R,r} \cup \hat\tau_j^{-1} \Omega_{\epsilon_1,R,r}\cup \Omega_{\epsilon}\cup \hat\tau_j^{1} \Omega_{\epsilon_1,R,r} \cup \hat\tau_j^{2} \Omega_{\epsilon_1,R,r}.
$$

 We recall that $\Om_{\epsilon_j,r_j}\Subset  \Om_{\epsilon_l,r_l}$ if $l<j$. Assuming by induction that (\ref{goal}) holds for all $m'<m$, we have
\begin{align}
	 \left\|[(I)]_m\right\|_{\epsilon_m,R,r_m}&\leq \sum_{|Q|=2}^{m}R'^{|Q|}\sum_{M'\in E_{Q,m}}\prod_{j=1}^d
	\left\|[v_j+\phi^v_{j}]_{m_{j,1}}\right\|_{\epsilon_m,R,r_m}\cdots \left\|[v_j+\phi^v_{j}]_{m_{j,q_j}}\right\|_{\epsilon_m,R,r_m}
	\nonumber\\
	&\leq \sum_{|Q|=2}^{m}R'^{|Q|}\sum_{M'\in E_{Q,m}}\prod_{j=1}^d\eta_{m_{j,1}}A_{m_{j,1}}\cdots \eta_{m_{j,q_j}} A_{m_{j,q_j}}\nonumber\\
	&\leq \left[\sum_{|Q|=2}^{m}\eta_{Q,m}R'^{|Q|}(t+J^{m-1}(A(t))^{|Q|}\right]_m\nonumber\\
	&\leq  E_{m}[g_m(t)]_m,\label{estim-h'}
\end{align}
where we have set
	\begin{align*}
		\label{g(t)}
		&\eta_{Q,m}:=
		\max
		_{M'\in E_{Q,m}}\left(\prod_{i=1}^d\eta_{m_{i,1}}\cdots \eta_{m_{i,q_i}}\right),\quad E_m:=\max_{\dindice{Q\in \Bbb N^d}{2\leq |Q|\leq m}} \eta_{Q,m},\\
		& g_m(t):= \sum_{|Q|=2}^{m}R'^{|Q|}(t+J^{m-1}(A(t))^{|Q|},\quad   g(t):= \sum_{|Q|\geq 2}R'^{|Q|}(t+A(t))^{|Q|}.
			\end{align*}
			Here $[g(t)]_m$ denotes the coefficient of $t^m$ in the power series $g(t)$.
			We also define
		\begin{align*}
				&  G(t,U):= \sum_{|Q|\geq 2}R'^{|Q|}(t+U)^{|Q|},\\
			& g^{\pm e_i}(t):= \sum_{|Q|\geq 2}R'^{|Q|}(t+B^{\pm e_i}(t))^{|Q|}=G(t,B^{\pm e_i}(t)),\\
				& g_m^{\pm e_i}(t):= \sum_{|Q|=2}^{m}R'^{|Q|}(t+J^{m-1}(B^{\pm e_i}(t))^{|Q|}.
			\end{align*}
We have (omitting the variables $x,y$)
\begin{eqnarray*}
	[(II)]_m &=&
	\sum_{\dindice{P\in \Bbb N_1^n}{m_1+m_2=m}}\frac{1}{P!}\left[ \partial^P_h \phi^v(T_ih,M_iv)\right]_{m_1}\left[\left(\tau_{i}^{*,h}(h,v+\phi^v(h,v))\right)^P\right]_{m_2}\\
	&=&
\sum_{\dindice{P\in \Bbb N_1^n}{m_1+m_2=m}}\frac{1}{P!} (\partial^P_h \left[\phi^v\right]_{m_1})(T_ih,M_iv)\left[\left(I\right)^P\right]_{m_2}
\end{eqnarray*}
Here, both indices $m_1$ and $m_2$ are $\geq 2$  so that both $m_1$ and $m_2$ are less or equal than $m-2$. Assuming by induction that (\ref{goal}) holds for all $m'<m$, we have
\begin{equation}\label{h'p}\nonumber
	\left\|\left[\left(I\right)^P\right]_{m_2}\right\|_{\epsilon_{m-1},R,r_{m-1}}\leq E_{m_2}\left[\left(\sum_{|Q|=2}^{m_2}R^{|Q|}(t+J^{m-1}(A(t))^{|Q|}\right)^{|P|}\right]_{m_2}.
\end{equation}
Indeed,
\begin{eqnarray*}
	\left[\left(I\right)^P\right]_{m_2}&=&\left[\prod_{i=1}^{n-a-b}((I)_{i})^{p_i}\right]_{m_2}\\
	&=&  \sum_{\sum_i(m_{i,1}+\cdots +m_{i,p_i})=m_2}\prod_{i=1}^{n-a-b}[(I)_{i}]_{m_{i,1}}\cdots [(I)_{i}]_{m_{i,p_i}}.
\end{eqnarray*}
Here,$(I)_{i}$ means the $i-$th component of term $(I)$.
According to $(\ref{concatenation})$ and by $(\ref{estim-h'})$, we have
\begin{align*}
 \left\|\prod_{i=1}^{n-a-b}[(I)_{i}]_{m_{i,1}}\cdots [(I)_{i}]_{m_{i,p_i}}\right\|_{\epsilon_{m-1},R,r_{m-1}}\leq &\\ \prod_{i=1}^{n-a-b}E_{m_{i,1}}\left[g_{m_{i,1}}(t)\right]_{m_{i,1}}\cdots E_{m_{i,p_i}}\left[g_{m_{i,p_i}}(t)\right]_{m_{i,p_i}}&\\
	\leq  \max_{2\leq |Q|\leq m_{2}}\eta_{Q, m_{2}}\prod_{i=1}^{n-a-b}\left[g_{m_{i,1}}(t)\right]_{m_{i,1}}\cdots \left[g_{m_{i,p_i}}(t)\right]_{m_{i,p_i}}.&
\end{align*}
Hence, we have
$$
\sum_{\sum_i(m_{i,1}+\cdots +m_{i,p_i})=m_2} \left\|\prod_{i=1}^{n-a-b}[(I)_{i}]_{m_{i,1}}\cdots [(I)_{i}]_{m_{i,p_i}}\right\|_{\epsilon_{m-1},R,r_{m-1}}\leq  E_{m_{2}}[g(t)^{|P|}]_{m_2}.
$$

Now we estimate $\left[ \partial^P_h \phi^v(T_ih,M_iv)\right]_{m_1}$ where attention is put on the choice of domain.
By Cauchy estimate $(\ref{Reps3})$, we have  by induction on $m\geq 2$, for any $m_1 < m$,
$$
||\left[ \partial^P_h \phi^v(T_ih,M_iv)\right]_{m_1}||_{\epsilon_{m_1},R, r_{m_1}}
\leq \frac{C (C')^{|P|}(\sum_{j,\pm} B^{\pm e_j}_{m_1}\eta_{m_1}+A_{m_1}\eta_{m_1})}{C''^{\nu+|P|}}
$$
where the $L^\infty-$norm over the domain $\hat\tau_i(\Omega_{\epsilon_{m_1}+\eta,R,r_{m_1}})$ follows from the fact that
 $\hat\tau_i(\Omega_{\epsilon_{m_1}+\eta,R,r_{m_1}})$
is contained in the convex hull of the union, over $j$, of the union of
$\hat\tau_j(\Omega_{\epsilon_{m_1},R,r_{m_1}})\cup \hat\tau_j^2(\Omega_{\epsilon_{m_1},R,r_{m_1}})$
(related to the coefficient of $B^{e_j}(t)$),
$\hat\tau_j^{-1}(\Omega_{\epsilon_{m_1},R,r_{m_1}}) \cup \hat\tau_j^{-2}(\Omega_{\epsilon_{m_1},R,r_{m_1}})$
(related to the coefficient of $B^{-e_j}(t)$)
and $\Omega_{\epsilon_{m_1},R,r_{m_1}} \cup\cup_{k=1}^q \hat\tau_k(\Omega_{\epsilon_{m_1},R,r_{m_1}})\cup\cup_{k=1}^q \hat\tau_k^{-1}(\Omega_{\epsilon_{m_1},R,r_{m_1}})$
(related to the coefficient of $A(t)$).
Here $C''$ is a constant independent of $m$. As a consequence, we have
\begin{align}
	||\left[\left(II\right)\right]_{m}||_{\epsilon_{m-1}, R,r_{m-1}}\leq &\sum_{m_1+m_2=m}\sum_{\dindice{P\in \Bbb N^{n-a-b}}{|P|\geq 1}}
	\frac{C (C')^{|P|}(A_{m_1}+\sum_{j,\pm} B^{\pm e_j}_{m_1})\eta_{m_1}}{C''^{\nu+|P|}}E_{m_{2}}[g(t)^{|P|}]_{m_2}\nonumber\\
	\leq & \sum_{m_1+m_2=m}\frac{C (A_{m_1}+\sum_{j,\pm} B^{\pm e_j}_{m_1})\eta_{m_1}}{C''^\nu} \left[E_{m_{2}}\sum_{\dindice{P\in \Bbb N^{n-a-b}}{|P|\geq 1}}
	\left(C' {g(t)}/C''\right)^{|P|}\right]_{m_2}\nonumber\\
	\leq & \frac{C }{C''^\nu} \left(\max_{m_1+m_2=m}\eta_{m_1}E_{m_2}\right)\times\nonumber\\
	&\times\left[( A(t)+\sum_{j,\pm} B^{\pm e_j}(t)) \left(\left(\frac{1}{1-C' g(t)/C''}\right)^{n-a-b}-1\right)\right]_m.\label{estim-h''}
\end{align}
Collecting estimates $(\ref{estim-h'})$ and $(\ref{estim-h''})$, we obtain
\begin{align*}\nonumber
		\left\|L^v_i[\phi^v]_m 	\right\|_{\epsilon_{m-1}, R,r_{m-1}}& \leq\left[ E_m g(t)+ \frac{C }{C''^\nu} \left(\max_{m_1+m_2=m}\eta_{m_1}E_{m_2}\right)\times\right.\\
	&\left.\times(A(t)+\sum_{j,\pm} B^{\pm e_j}(t)) \left(\left(\frac{1}{1-C' g(t)/C''}\right)^{n-a-b}-1\right)\right]_m.
\end{align*}
We solve the vertical cohomological operator for $(\ref{linear-def})$ and obtain by Proposition \ref{cohomo} the following estimate~:
\begin{align}
			\left\|[\phi^v]_m	\right\|_{\epsilon_{m},R, r_{m}}&\leq \frac{ C_1}{\eta^{\tau+\nu}} 2^{m(\tau+\nu)} \left[ E_m g(t)+ \frac{C }{C''^\nu} \left(\max_{m_1+m_2=m}\eta_{m_1}E_{m_2}\right)\right.\nonumber\\
	&\left.\times (A(t)+\sum_{j,\pm} B^{\pm e_j}(t)) \left(\left(\frac{1}{1-C' g(t)/C''}\right)^{n-a-b}-1\right)\right]_m.\label{induction}
\end{align}
Note that we use that $\hat\tau_i\hat\tau_j=\hat\tau_j\hat\tau_i$ to apply Proposition \ref{cohomo} to $L_i^v(\phi^v)$.
Using definition $(\ref{def-eta})$, we obtain
\begin{align*}\nonumber
		\left\|[\phi^v]_m 	\right\|_{\epsilon_{m},R, r_{m}}\leq \eta_m \left[  g(t)+ \frac{C }{C''^\nu}  (A(t)+\sum_{j,\pm} B^{\pm e_j}(t))\right.\times&\\
		 \left.\left(\left(\frac{
		1}{1-C' g(t)/C''}\right)^{n-a-b}-1\right)\right]_m.&
\end{align*}

For each $1\leq i\leq q$, we consider the single inverse vertical cohomological equation (\ref{linear-}). By Proposition \ref{str-cohomo} and Remark \ref{i-domain}, we obtain ,
\begin{align*}\nonumber
	\left\|[\phi^v]_m 	\right\|_{\hat\tau_i^2(\Omega_{\epsilon_{m},R, r_{m}})}\leq \eta_m \left[  g^{+e_i}(t)+ \frac{C }{C''^\nu}  (A(t)+\sum_{j,\pm} B^{\pm e_j}(t))\right.\times&\\
	\left.\left(\left(\frac{1}{1-C' g^{+e_i}(t)/C''}\right)^{n-a-b}-1\right)\right]_m.&
	\end{align*}
By (\ref{estim-i+2}), we also obtain the same estimate for $	\left\|[\phi^v]_m 	\right\|_{\hat\tau_i(\Omega_{\epsilon_{m}, R,r_{m}})}$.
Similarly, considering, for each $1\leq i\leq q$, the single vertical cohomological equation (\ref{linear+}), by Proposition  \ref{str-cohomo} and Remark \ref{i-domain}, we also have
\begin{align*}\nonumber
		\left\|[\phi^v]_m	\right\|_{\hat\tau_i^{-2}(\Omega_{\epsilon_{m},R, r_{1,m}})}\leq \eta_m \left[  g^{-e_i}(t)+ \frac{C }{C''^\nu}  (A(t)+\sum_{j,\pm} B^{\pm e_j}(t)) \right.&\\
		\left.\left(\left(\frac{
		1}{1-C' g^{-e_i}(t)/C''}\right)^{n-a-b}-1\right)\right]_m.&
\end{align*}
By (\ref{estim+i-2}), we obtain the same estimate for $	\left\|[\phi^v]_m 	\right\|_{\hat\tau_i^{-1}(\Omega_{\epsilon_{m},R, r_{m}})}$.

Let us consider the functional equation system, $i=1,\ldots, q$,
\begin{equation}\label{equ-A}\nonumber
	A(t)=   G(t,A(t))+ \frac{C }{C''^\nu}  (A(t)+\sum_{j,\pm} B^{\pm e_j}(t)) \left(\left(\frac{
		1}{1-C' G(t,A(t))/C''}\right)^{n-a-b}-1\right),
\end{equation}
\begin{equation}\label{equ-A'}\nonumber
	B^{\pm e_i}(t)=   G(t,B^{\pm e_i}(t))+ \frac{C }{C''^\nu}  (A(t)+\sum_{j,\pm} B^{\pm e_j}(t)) \left(\left(\frac{
		1}{1-C' G(t,B^{\pm e_i}(t))/C''}\right)^{n-a-b}-1\right).
\end{equation}
This equation system has a unique analytic solution vanishing at the origin at order $2$ as shown by the implicit function theorem.
Notice that the coefficients of the powers of $A(t)$,  $B^{\pm e_j}(t)$ are non-negative. As $A_2=B_2^{\pm e_i}=[G(t,0)]_2> 0$, we obtain by induction that all coefficients of degree $m\geq 2$ of  $A(t)$ and  $B^{\pm e_j}(t)$ are non-negative.



We now can prove the theorem. Indeed by assumption, there are positive constants $M'',L$ such that $\eta_m\leq M''L^m$ for all $m\geq 2$. Since $A(t)$ converges at the origin, then $A_m\leq D^m$ for some positive $D$. By the majorant construction and previous estimates, we have $\left\|[\phi^v]_m 	\right\|_{\epsilon_{m},R, r_{m}}\leq A_m$ for all $m\geq 2$. Hence, according to (\ref{minorants}), we have $$||[\phi^v]_m||_{\frac{\epsilon_1}{2},R,r_1e^{-1}}\leq M''(DL)^m$$ for all $m\geq 2$. Hence, $\phi^v$ converges near the Lie group.

This proves the theorem.
\subsection{Full linearization for embedded toroidal manifolds}
\begin{mythm}\label{full-toroidal}
	Let $C$ be an $n$-dimensional complex abelian group, holomorphically embedded into a complex manifold $M_{n+d}$.   Assume that $T_{M|_C}$ splits. Assume the normal bundle $N_C$ has (locally constant) unitary transition functions. Assume that  $N_C$ is fully Diophantine (see Definition (\ref{dioph_full})). Then $(M,C)$ is biholomorphic to $(N_{C/M},C)$ as germ.
\end{mythm}
\begin{proof}
If we can find $\Phi= \mathrm{Id}+\phi$ be a biholomorphism of  $\Omega_{\epsilon_0,r}$ (to be chosen) such that for any $1 \leq i \leq q$,
\begin{equation}
\Phi \circ \hat \tau_i =\tau_i \circ \Phi \label{Phi}
\end{equation}
for some biholomorphism of $\Omega_{\epsilon_0,r}$ (to be chosen with $r$ decreasing)
$$ \hat \tau_i(x,y,h,v)=( x,y,T_i h, M_iv)$$
such that $(M,C)$ is biholomorphic to the quotient of $\Omega_{\epsilon_0,r}$ by  $\tau_1,\dots, \tau_q$.
In that case, $(M,C)$ is biholomorphic to the quotient of $\Omega_{\epsilon_0,r}$ by  $\hat \tau_1,\dots, \hat\tau_q$ which is biholomorphic to $(N_{C/M},C)$.

Define the higher order perturbations
$$ \tau_{i,\pm}^{*,h}:= \tau_{i, \pm}^{h}-T_i^{\pm 1},\tau_{i, \pm}^{*,v}:= \tau_{i, \pm}^{v}-M_i^{\pm 1}.$$
The horizontal part of equation $(\ref{Phi})$ is given by
\begin{equation}
T_i^{\pm 1} h+ \phi^h(x,y,T_i^{\pm 1} h,M_i^{\pm 1} v)=T_i^{\pm 1} h+T_i^{\pm 1} \phi^h+\tau_{i, \pm}^{*,h}(x,y,h+\phi^{h}(x,y,h,v),v+\phi^{v}(x,y,h,v)),
\end{equation}
that is
\begin{equation}\label{horizontal_full}
\phi^h(x,y,T_i^{\pm 1} h,M_i^{\pm 1} v)=T_i^{\pm 1} \phi^h+\tau_{i, \pm}^{*,h}(x,y,h+\phi^{h}(x,y,h,v),v+\phi^{v}(x,y,h,v)),
\end{equation}
The vertical part of equation $(\ref{Phi})$ is given by
\begin{equation}
M_i^{\pm 1} v+ \phi^v(x,y,T_i^{\pm 1} h,M_i^{\pm 1} v)=M_i^{\pm 1} v+M_i^{\pm 1} \phi^v+\tau_{i, \pm}^{*,v}(x,y,h+\phi^{h}(x,y,h,v),v+\phi^{v}(x,y,h,v)),
\end{equation}
that is
\begin{equation}\label{vertical_full}
 \phi^v(x,y,T_i^{\pm 1} h,M_i^{\pm 1} v)=M_i^{\pm 1} \phi^v+\tau_{i, \pm}^{*,v}(x,y,h+\phi^{h}(x,y,h,v),v+\phi^{v}(x,y,h,v)).
\end{equation}
The cohomological operator is defined as
in $(\ref{linear-def-total})$.
In particular, we need to solve the cohomological operator
\begin{equation}
\label{linear_full}
 L_{i, \pm}(\phi)=\tau_{i, \pm}^{*}(x,y,h+\phi^{h}(x,y,h,v),v+\phi^{v}(x,y,h,v)).
\end{equation}

Note that there exists a formal (power series in $v$, with holomorphic coefficients in some $\Omega_{\epsilon',r}$) solution to both (\ref{horizontal_full}) and (\ref{vertical_full}).
More precisely, as in the previous theorem,
we proceed by induction on $m\geq 2$ as
	Taylor expansion at $v=0$ of (\ref{linear_full}) shows that for any $m \geq 2$,
	\begin{equation}\label{degm}L_{i, \pm} [\phi]_m=P_i(x,y,h; v, [\phi]_2, \cdots, [\phi]_{m-1})\end{equation}
	where $P_i(x,y,h; v, [\phi]_2, \cdots, [\phi]_{m-1})$ is analytic in $x,y,h$ and polynomial in $v$, $[\phi]_2$, ..., $[\phi]_{m-1}$.
As in the previous theorem \ref{ueda-thm}, we will estimate the $L^\infty$ norm to show that this formal solution is in fact convergent.
To do so, we will follow the majorant method in \cite[Section 4]{GS21}.

In the following, we will estimate $[\phi^v]_k(k \geq 2)$ and $[\phi^h]_k(k \geq 2)$ by induction on $k$ (which gives the estimate for $\phi^v, \phi^h$).
Compared to the previous theorem, we only need to study the cohomological operator.

Let $0<r_1<\min(1,r)$ and $0<\epsilon_1<\epsilon_0$ be positive constants to be chosen below sufficiently small. Let us define the sequences  $r_{m+1}=r_me^{-\frac{1}{2^{m}}}$ and $\epsilon_{m+1}=\epsilon_m- \epsilon_1\frac{\eta}{ 2^{m}\kappa}$ for $m>0$ and some $\eta<\frac{\kappa}{2}$ sufficiently small. We have  $r_{m+1}:=r_1e^{-\sum_{k=1}^m\frac{1}{2^k}}$ and $\epsilon_{m+1}:= \epsilon_1-\epsilon_1\sum_{k=1}^m\frac{\eta}{2^k  \kappa}$ for $m\geq 1$.

Our goal to find germs of holomorphic function at 0
$$A(t)=\sum_{k \geq 2} A_k t^k,$$
such that
for any fixed $R>0$
\begin{equation}\label{goal_full}
\sup_{(x,y,h,v) \in\cup_{i=0}^q \hat\tau_i^{\pm 1}(\Omega_{\epsilon_k,R, r_k}) }| [\phi]_k(x,y,h,v)|  \leq A_k \eta_k,
\end{equation}
for suitable chosen $ r_1$.
Here $A_k, \eta_k$ depends on $R$.
 Here, the sequence $\{\eta_m\}_{m\geq 1}$ is defined as previous.
 We have
\begin{equation}
\eta_m
\leq D^m,
\end{equation}
for some positive constant $D$ as before.

Consider the Taylor development
$$\phi(x,y,h,v)=\sum_{Q \in \N^d, |Q| \geq 2} \phi_Q(x,y,h)v^Q.$$
Let $[\phi]_k$ be the homogeneous degree $k$ part of $\phi$
$$[\phi]_k(x,y,h,v)=\sum_{Q \in \N^d, |Q| =k} \phi_Q(x,y,h)v^Q.$$
In the following, we will always denote $[\bullet]_k$ to indicate the homogeneous degree $k$ part of some series in $v$.
Notice that
$$[L_{\pm i}(\phi)]_k=L_{\pm i} [\phi]_k.$$
The degree 2 part  is
$$L_i([\phi]_2)= [\tau_{i,+}^{*}(x,y,h,v)]_2,$$
whose right-handed-side term is independent of $\phi$.

According to the following equation
\begin{equation}\label{compat-full}
	L_i([\tau_{j,+}^{*}]_{m})= L_j([\tau_{i,+}^{*}]_{m})
	\end{equation}
and Proposition \ref{cohomo}, these two sets of equations have the same unique solution $[\phi]_2$ on $\Omega_{ \epsilon_1,r_1}$.

We proceed by induction on $m\geq 2$ as
	Taylor expansion at $v=0$ of (\ref{horizontal_full}) and (\ref{vertical_full}) shows that for any $m \geq 2$,
	\begin{equation}
	L_i [\phi]_m=P_i(x,y,h; v, [\phi]_2, \cdots, [\phi]_{m-1})\end{equation}
	where $P_i(x,y,h; v, [\phi]_2, \cdots, [\phi]_{m-1})$ is analytic in $x,y,h$ and polynomial in $v$, $[\phi]_2$, ..., $[\phi]_{m-1}$.
We further expand the first expression on the right-hand side as
\begin{eqnarray*}
	L_i [\phi]_m &=&
	\sum_{\dindice{P\in \Bbb N_1^{n-a-b},Q \in \N_2^d}{m_1+m_2=m}}\frac{1}{P!}\partial_h^P \tau_{i,Q}^{*}(x,y,h)\left[ \left(\phi^h\right)^P\right ]_{m_1}\left[\left(v+\phi^v\right)^Q\right]_{m_2}.
\end{eqnarray*}
where
$$\tau_{i}^{*}(x,y,h,v)=\sum_{Q \in \N_2^d} \tau_{i,Q}^{*}(x,y,h)v^Q.$$

By Cauchy estimate $(\ref{Reps})$ applying to $\tau_{i,j}^*(1 \leq j \leq n-a-b+d)$ implies that, if $\epsilon_1$ small enough, there exists $R'>0$, $M>0$ (depending on $R$) 
such that
$$
||\tau_{i,Q,j}^*||_{\epsilon_1,R,r} \leq R'^{|Q|},
$$
for any $P \in \Bbb N_1^{n-a-b}$
\begin{equation}\label{estimate_M}
 ||\partial_h^P\tau_{i,Q,j}^*||_{\epsilon_1,R,r} \leq \frac{R'^{|Q|}}{M^{|P|}},
\end{equation}
with
$$\tau_{i,j}^*=\sum_{Q \in \N^d} \tau^*_{i,Q,j} (x,y,h)v^Q$$
and $\tau_i^*=(\tau_{i,1}, \cdots, \tau_{i,n-a-b+d})$.
Note that the functions $\tau_{i,j}^*$ are independent of the induction process.
For $\epsilon_1$ small enough (independent of the induction process), we may assume that the distance between $\Omega_{\epsilon_1, R,r(R)}$ and the boundary of definition domain of $\tau_{i,j}^*$ has a uniform lower bound which induces the constant $M$ in the estimate $(\ref{estimate_M})$.

Assuming by induction that (\ref{goal_full}) holds for all $m'<m$, we have
\begin{equation}
 \left\|	\left[\left(v+\phi^v\right)^Q\right]_{m_2}\right\|_{\epsilon_{m-1},R,r_{m-1}}\leq E_{m_2}[(t+A(t))^{|Q|}]_{m_2}
\end{equation}
where we have set
	\begin{align*}
		\label{g(t)}
		&\eta_{Q,m}:=
		\max
		_{M'\in E_{Q,m}}\left(\prod_{i=1}^d\eta_{m_{i,1}}\cdots \eta_{m_{i,q_i}}\right),\quad E_m:=\max_{\dindice{Q\in \Bbb N^d}{2\leq |Q|\leq m}} \eta_{Q,m},
			\end{align*}

A similar calculation gives
\begin{equation}
 \left\|	\left[\left(\phi^h\right)^P\right]_{m_1}\right\|_{\epsilon_{m-1},R,r_{m-1}}\leq E_{m_1}[A(t)^{|P|}]_{m_1}.
\end{equation}

As a consequence, we have
\begin{align}
	||L_i [\phi]_m||_{\epsilon_{m-1}, R,r_{m-1}}\leq &\sum_{m_1+m_2=m}\sum_{\dindice{P\in \Bbb N_1^{n-a-b}}{Q\in \Bbb N_2^d}}
	\frac{(R')^{|Q|}}{M^{|P|}}E_{m_1}[A(t)^{|P|}]_{m_1}E_{m_2}[(t+A(t))^{|Q|}]_{m_2}\nonumber\\
	\leq & \sum_{m_1+m_2=m}
	E_{m_1}\left[h(t)\right]_{m_1}E_{m_2}\left[g(t)\right]_{m_2}\nonumber\\
	\leq &E_m [h(t)g(t)]_m.
\end{align}
where we have set
\begin{align}
g_m(t):= \sum_{|Q|=2}^{m}(R')^{|Q|}(t+J^{m-1}(A(t))^{|Q|},\quad   g(t):= \sum_{|Q|\geq 2}(R')^{|Q|} (t+A(t))^{|Q|}.
\end{align}
\begin{align}
h_m(t):= \sum_{|P|=2}^{m}\frac{1}{M^{|P|}}(J^{m-1}(A(t))^{|P|},\quad   h(t):= \sum_{|P|\geq 2}\frac{1}{M^{|P|}} A(t)^{|P|}.
\end{align}
			Here $[g(t)]_m$,$[h(t)]_m$ denotes the coefficient of $t^m$ in the power series $g(t)$, $h(t)$.

The rest of the proof is completely as the previous proof.
Let us consider the functional equation system, $i=1,\ldots, q$,
\begin{equation}\label{equ-full}
	A(t)=   g(t)h(t).
\end{equation}
This equation system has a unique analytic solution vanishing at the origin at order $2$ as shown by the implicit function theorem.
Notice that the coefficients of the powers of $A(t)$ are non-negative by induction.



We now can prove the theorem. Indeed by assumption, there are positive constants $M'',L$ such that $\eta_m\leq M''L^m$ for all $m\geq 2$. Since $A(t)$ converges at the origin, then $A_m\leq D^m$ for some positive $D$. By the majorant construction and previous estimates, we have $\left\|[\phi]_m 	\right\|_{\epsilon_{m},R, r_{m}}\leq A_m$ for all $m\geq 2$. Hence, we have $$||[\phi]_m||_{\frac{\epsilon_1}{2},R,r_1e^{-1}}\leq M''(DL)^m$$ for all $m\geq 2$. Hence, $\phi$ converges near the Lie group.

This proves the theorem.
\end{proof}
\section{Neighborhoods of embedded Hopf manifolds}

In this section we consider, Hopf manifolds embedded as hypersurfaces into another complex manifold. 
Our aim is to give sufficient condition ensuring that such an embedded Hopf manifold has a neighborhood  biholomorphic to a neighborhood of the zero section into its normal bundle. In the article \cite{Tsu84}, H. Tsuji considered this "full linearization problem" for a Hopf surface embedded as a hypersurface when the normal bundle is a flat line bundle over the Hopf surface.

Let us recall some basic results of Hopf manifolds following \cite{Has93}.
The generalized definition above is due to Kodaira \cite{Kod66}.
\begin{mydef}\label{def-hopf}
	A Hopf manifold is a compact complex manifold of which the universal covering
	is $\C^n\setminus \{0\}$, where $n$ is a positive integer $(n \geq 2)$.
	A primary Hopf manifold is a compact complex manifold of
	which the covering transformation group is an infinite cyclic group.
	It is called of linear type if the covering transformation group is generated
	by a linear contraction of Jordan form.
	It is called of diagonal type if the covering transformation group is generated
	by a diagonalizable linear contraction.
\end{mydef}
It is well known that a Hopf manifold is the first example of a non-K\"ahler compact
complex manifold, that is, it does not admit a K\"ahler structure compatible with its
complex structure violating the topology constraint.
A primary Hopf manifold is a complex structure on $S^1 \times S^{2n-1}$.

\begin{mydef}\label{def-generic}
	A Hopf manifold $C$ is called of generic type if it is generated by a
	contraction of the type $\varphi: (z_1
	,\cdots,z_n) \to (\alpha_1 z_1, \cdots, \alpha_n z_n)$ with $0<|\alpha_1 |< \cdots
	< |\alpha_n | <1$, and there are no relations except trivial ones between the $\alpha_i$ of the
	form
	$$ \prod_{i \in A} \alpha_i^{r_i} = \prod_{j \in \{1, \cdots, n\} \setminus A} \alpha_j^{r_j}, r_i \in \N
	$$
	for any $A \subset \{1, \cdots, n\}$.\\
	Let $\Delta_{\alpha_1, \cdots, \alpha_n}$ be the free commutative subgroup of $\C^*$ generated by the eigenvalues
	$\alpha_1, \cdots, \alpha_n$.
\end{mydef}
Note that the cardinality of $\Delta_{\alpha_1, \cdots, \alpha_n}$
is at most countable.
Note that a Hopf manifold of linear type which can be deform to a Hopf manifold of generic type is necessarily of generic type.

Now we can state the main result on the full linearization of embedded Hopf manifolds.
\begin{mythm}
	\label{hopf}
	Let $C$ be a generic type Hopf manifold embedded as hypersurface in some complex manifold $X$.
	Assume that $N_{C/X}$ is a non-Hermitian-flat (flat) line bundle over $C$ corresponding to $\beta \in \C^*$. 
	Assume that $\beta^{ m}  \notin  \Delta_{\alpha_1, \cdots, \alpha_n}$ for any $m \geq 1$.
	Then $(C,X)$ and $(C, N_{C/X})$ are holomorphically equivalent.
\end{mythm}

\subsection{Properties of Hopf manifolds}
In \cite{Tsu84}, Tsuji considered the full linearization problem in a neighborhood of a Hopf surface embedded as a hypersurface when the normal bundle is a flat line bundle over the Hopf surface.
One of the key point in \cite[Lemma 2.33]{Tsu84} is that line bundles he considered as normal bundles are not Hermitian flat. These line bundles may exist only if the manifold is not compact K\"ahler as shown in the following classical result~:
\begin{mylem}
	Let $L$ be a flat line bundle over a compact K\"ahler manifold $X$.
	Then $L$ is Hermitian flat.
\end{mylem}
\begin{proof}
	Let $h_0$ be an arbitrary Hermitian metric on $L$. Its Chern curvature is given by
	$$
	\Theta(h_0) = -\partial\bar{\partial}\log h_0.
	$$
	Since $L$ is flat, the first Chern class $c_1(L)$ vanishes in de Rham cohomology. By the $\partial\bar{\partial}$-lemma, the curvature form is therefore $\partial\bar{\partial}$-exact; that is, there exists a smooth function $\varphi$ on $X$ such that
	$$
	\Theta(h_0) = \partial\bar{\partial}\varphi.
	$$
	
	Define a new metric $h$ on $L$ by twisting $h_0$ with the potential $\varphi$:
	$$
	h = e^{-\varphi} h_0.
	$$
	Then, the Chern curvature of $h$ is computed as follows:
	$$
	\Theta(h) = -\partial\bar{\partial}\log h = -\partial\bar{\partial}(\log h_0 - \varphi)
	= -\partial\bar{\partial}\log h_0 + \partial\bar{\partial}\varphi.
	$$
	Substituting the expression for $\Theta(h_0)$ gives:
	$$
	\Theta(h) = \Theta(h_0) - \Theta(h_0) = 0.
	$$
	
	Thus, $h$ is a Hermitian metric on $L$ with trivial curvature, which shows that $L$ is Hermitian flat.
\end{proof}
Another property of Hopf surfaces that Tsuji used is the fact that a primary Hopf surface (see Definition \ref {def-hopf}) has many holomorphic tangent fields. Although in general, Poincaré–Hopf index theorem implies the existence of zeros of tangent fields, it is however enough to have many holomorphic tangent fields non vanishing over some measure 0 closed set. That's why, in Proposition \ref{Shilov}, we assume that the tangent fields do not vanish near the Shilov boundary.
\begin{mydef}[Shilov boundary]
	Let \(U\subset \mathbb{C}^n\) be a bounded open set and consider the uniform algebra
	\[
	A(U)=\{ f\in C^0(\overline{U}) \mid f \text{ is holomorphic on } U\}.
	\]
	The \emph{Shilov boundary} \(\Gamma_A\) of \(A(U)\) (which is contained in the boundary of $U$ with respect to the usual Euclidean topology) is defined as the unique smallest closed subset $\Gamma_A \subset \overline{U}$, such that for every \(f\in A(U)\),
	\[
	\max_{z\in \overline{U}} |f(z)| = \max_{z\in \Gamma_A} |f(z)|.
	\]
\end{mydef}
\begin{myex}
	\begin{itemize}
		\item If $U = B^n = \{ z\in \mathbb{C}^n \mid \|z\| < 1\}$, then 
		$\Gamma_A = \{ z\in \mathbb{C}^n \mid \|z\|=1\}$.
		\item If $U = \mathbb{D}^n = \{ (z_1,\dots,z_n) \in \mathbb{C}^n \mid |z_j|<1 \text{ for } j=1,\dots,n \}$, then $\Gamma_A = \{ (z_1,\dots,z_n) \in \overline{\mathbb{D}}^n \mid |z_1|=\cdots=|z_n|=1\}$.
		This is obtained by applying the maximum principle
			with respect to each complex variable.
Note that for any $w \in \partial \mathbb{D} $, any $f \in A(\mathbb{D}^n)$ and any $a \in \N$,
the restriction $f|_{\mathbb{D}^a \times \{w\} \times \mathbb{D}^{n-1-a}} \in A(\mathbb{D}^a \times \{w\} \times \mathbb{D}^{n-1-a})$
by Montel's theorem.
			Same arguments work for products of domains in $\C$.
	\end{itemize}
\end{myex}
We have the following proposition.
\begin{myprop}
	\label{Shilov}
	Let \(U\subset \mathbb{C}^n\) be a bounded domain.
	Consider holomorphic tangent fields $Z_1, \cdots, Z_n$ defined on some neighborhood of $\bar{U}$.
	Assume that $Z_1 \wedge \cdots \wedge Z_n$ is nowhere vanishing on the Shilov boundary $\Gamma_A$.
	Then there exists constant $C>0$ such that
	\begin{equation}
		\sum_{i=1}^n || \d_i f|| \leq C(\sum_{j=1}^n ||Z_j f||) \label{Z}
	\end{equation}
	for any $f \in A(U)$ holomorphic in some neighborhood of $\bar{U}$
	where $||f||$ means the $L^\infty$ norm of $f$ on $U$.
\end{myprop}
\begin{proof}
	Note that for any $i$, $\d_i f$ is also in $A(U)$.
	Write
	$$Z_i f= \sum_{j} g_{ij} \d_j f$$
	for some holomorphic functions $g_{ij}$ defined on some neighborhood of $\bar{U}$.
	By assumption, for any $z \in \Gamma_A$,
	$$C^{-1} \leq |det(g_{ij}(z))| \leq C$$
	for some $C>0$.
	Thus for any $i$, for some possible $C>0$,for any $z \in \Gamma_A$,
	$$|\d_i f|(z) \leq C(\sum_j |Z_j f|(z)).$$
	By definition of Shilov boundary,
	$$||\d_i f|| \leq C(\sum_j ||Z_j f||).$$
\end{proof}
We shall use the definition in a broader sense that the one given in \cite{GS21} as we do not require $U_j^r$ to be be biholomorphic to a polydisc (see Section \ref{nested-hopf}).
\begin{mydef}\cite[Definition A.1]{GS21}\label{nested-cov}
	Let $\{U_j^r\}$ be an open covering of a complex manifold $C$ for each $r\in[r_*,r^*]$. We say that the family of coverings $\{U_j^r\}$ is {\it nested}, if each connected component of $U_{k}^{\rho}\cap  U_{j}^{r_*}$ intersects $U_{k}^{r_*}\cap  U_{j}^{r_*}$ when $r_*\leq \rho\leq r^*$.
	In particular, $U_{k}^{r_*}\cap  U_{j}^{r_*}$ is non-empty if and only if $U_{k}^{\rho}\cap  U_{j}^{r_*}$ is non-empty.
\end{mydef}

Now we can state the variant of \cite[Lemma 2.33]{Tsu84}.
\begin{myprop}
	\label{tsu_lem}
	Let $C$ be a compact complex manifold of dimension $n$.
	Let $\cU^r = \{U_j^r \}$, $r_* \leq r < r^*$, be family of nested coverings of $C$. Let $F,L$ be holomorphic flat line bundles over $C$.
	We assume~:
	\begin{enumerate}
		\item $L$ is not Hermitian flat, the transition functions $l_{ij}$ of which satisfies~:
		Assume that for any $i_0$, there exist $A(i_0)\in \N$, $N_0=i_0, N_1, \cdots, N_{A(i_0)}$
		such that
		\begin{equation}\label{lb1}
			|l_{N_0 N_1}| \leq 1, |l_{N_1 N_2}| \leq 1, \cdots, |l_{N_{A(i_0)-2} N_{A(i_0)-1}}| \leq 1, |l_{N_{A(i_0)-1} N_{A(i_0)}}| < 1.
		\end{equation}
		\item for any $m \in \N^*$,$H^1(X, F+mL)=0.$
		\item there exists holomorphic vector fields $Z_1, \cdots, Z_n$
		such that $Z_1 \wedge \cdots \wedge Z_n$
		is nowhere vanishing on the Shilov boundary of any open set in the cover $\cU^r$.
	\end{enumerate}
	Then we have for any $r_*<r''<r'<r^*$, for some $K>0$ independent of $m$, for any $\gamma \in C^0(\cU^{r'}, F+mL)$
	\begin{equation}\label{hopft-cohom-bound}
		||\gamma||_{r''} \leq K ||\delta \gamma||_{r'}.
	\end{equation}
\end{myprop}
\begin{proof}
	For any $q\in\N$, we define a linear endomorphism of $C^q(\cU^{r'}, F+mL)$ by
	$$Z_i: \eta=\{\eta_{i_0 \cdots i_q}\} \mapsto Z_i \eta:=\{Z_i \eta_{i_0 \cdots i_q}\}.$$
	Since $F,L$ are flat, this map is a cochain map as it commutes with the coboundary operators. Define for any multi-index $I=(I_1, \cdots, I_n)$,
	$$Z^I:=Z_1^{I_1} \cdots Z_n^{I_n}.$$
	Cauchy estimate implies that, for any $r''<r'$, there exists a positive constant $d$ such that
	for any $\gamma_m \in Z^1(\cU^{r'}, F+mL)$, any multi-index $I$,
	\begin{equation}\label{ZI}
		||Z^I \gamma_m||_{r''} \leq d^{|I|} || \gamma_m||_{r'}.
	\end{equation}
	The vanishing of cohomology class implies that
	the exists a positive constant $K_m$ such that
	for any $\tilde{\gamma}_m \in C^0(\cU^{r'}, F+mL)$,
	$$|| \tilde{\gamma}_m||_{r''} \leq K_m || \delta \tilde{\gamma}_m||_{r'}.$$
	Define
	$$E(m):=\left\{(\gamma_m,R) \}\in Z^1(\cU^{r'}, F+mL) \times \R; \forall I\in \N^n,||Z^\gamma_m||_{r''} \leq d^{|I|} R\right\}.$$
	In particular, according to (\ref{ZI}), we have
	$(\gamma_m,  || \gamma_m||_{r'}) \in E(m)$ for any $\gamma_m \in Z^1(\cU^{r'}, F+mL)$.
	Define
	$$K'_m=\sup \{|| \eta_m||_{r''}/R, \delta \eta_m=\gamma_m ,(\gamma_m,R) \in E(m)\}.$$
	We have $K_m \leq K'_m$. 
	In order to prove (\ref{hopft-cohom-bound}), it is enough to show that $K'_m$ is uniformly bounded with respect to $m$. Suppose it is not the case, then
	up to restricting to a subsequence, we many assume that $K'_m \to \infty$ as $m \to \infty$.
	and that there exist $\eta_m \in C^0(\cU^{r''}, F+mL)$ and $(\gamma_m, r_m)\in E(m)$
	such that
	$$\delta \eta_m=\gamma_m,\quad || \eta_m||_{r''}/r_m \leq K'_m \leq 2 || \eta_m||_{r''}/r_m,\quad || \eta_m||_{r''}=1.$$
	Thus $\lim_{m \to \infty} r_m=0$.
	Note that for any $i$,
	$(Z_i \gamma_m,dr_m) \in E(m).$
	Thus we have
	$$||Z_i \eta_m||_{r''} \leq K'_m d r_m \leq  2 || \eta_m||_{r''}/r_m \times dr_m=2d.$$
	By (\ref{Z}) in Proposition \ref{Shilov},
	the first derivatives of the \v{C}ech representatives $\eta_m$ in any coordinate chart are uniform bounded.
	In particular, up to taking some subsequence, we may assume that $\eta_m$ converges to a \v{C}ech cochain $h$ when $ m \to \infty$ by Arzela-Ascoli theorem.
	
	We claim that the \v{C}ech cochain so obtained vanishes identically. This is impossible since
	$|| \eta_m||_{r''}=1$ and this yields a contradiction.
	Thus to finish the proof, it is enough to prove that $h=0$.
	Write $\eta_m$ as holomorphic functions $\eta_{m,i}$ on $U^{r''}_i$.
	It satisfies the equation by
	$\delta \eta_m=\gamma_m$, i.e. on $U^{r''}_{ij}$,
	\begin{equation} \label{tsu_lem1}
		\eta_{m,i}-l_{ij}^mf_{ij} \eta_{m,j}=\gamma_{m,ij}
	\end{equation}
	where $l_{ij},f_{ij}$ are the transition functions of $L$ and $F$ on $U^{r''}_{ij}$.
	By assumption, the transition functions are constant functions.
	Recall that
	$||\gamma_{m}||_{r''} \leq r_m \to 0$
	as $m \to \infty$.
	For any $i_0$,
	$|l_{N_{A(i_0)-1} N_{A(i_0)}}| < 1$.
	Hence
	taking $m \to \infty$ and considering $(i,j)=(N_{A(i_0)-1}, N_{A(i_0)})$,
	we obtain
	$h_{N_{A(i_0)-1}}=0$.
	By considering successively equation (\ref{tsu_lem1}) for
	$(i,j)=(N_{A(i_0)-2}, N_{A(i_0)-1}), \cdots, (i,j)=(N_0, N_0+1)$,
	we can conclude that
	$h_{i_0}=0$.
	Since $i_0$ is arbitrary, we finishes the proof of the claim.
\end{proof}
\begin{myrem}
	The same proof works if we change $F$ to be a holomorphic flat vector bundle.
\end{myrem}

The rest of the section is devoted to prove that conditions of Proposition \ref{tsu_lem} holds true under assumptions of Theorem \ref{hopf} for well chosen family of  nested covering, $L:=N_C$, $F=TX_{|C}$.

Note that the Levi-Civita connection with respect to the standard Hermitian metric on $\C^n \setminus \{0\}$ is compatible with linear transformations and this connection is flat.
Thus this connection descends to a flat connection on a primary Hopf manifold of linear type.
In particular, the holomorphic tangent bundle or a cotangent bundle of a primary Hopf manifold of linear type is flat.

In general,	the classification problem of Hopf manifolds up to isomorphisms corresponds to the Poincaré-Dulac normal form problem for conjugation classes of contractions in
the group of automorphisms of $\C^n$ fixing 0.
When such a contraction is {\it non resonant}\cite{arnold-geometrical}, then it is holomorphically linearizable . So that the associated primary Hopf surface is of linear type. In general, there are resonances which forbid the contraction to be linearizable (see e.g \cite{Ueda-contraction}). This induces difficulty to check the cohomology condition in Proposition \ref{tsu_lem} for an arbitrary Hopf manifold.

\subsubsection{Genericity of Hopf manifolds}
Recall the following construction of Hasegawa which shows that for any primary Hopf manifold $C$, there exists a deformation over a disc such that the general fiber is $C$ while its center fiber is a primary Hopf manifold of diagonal type (cf. Step 1 in the proof of \cite[Theorem 2.8]{Tsu84}).

An analytic
automorphism $\varphi$ of $\C^n\setminus \{0\}$ induced from a covering transformation of a Hopf manifold $C$ may be considered as an analytic automorphism of $\C^n$ which
fixes the origin (by Hartogs’ theorem).
The analytic
automorphism $\varphi$ is a contraction if the sequence $\{\varphi^n\}$
converges uniformly to 0 on any compact neighborhood of the origin, or equivalently
for any relatively compact neighborhoods $U, V$ of the origin, there exists $N \in \N$
such that $\varphi^n(U) \subset V$ holds for any $n \geq N $.
For an analytic automorphism $\varphi$ over $\C^n$ which fixes the origin $0$, we denote the
linear part of $\varphi$ (i.e. the Jacobian matrix $d \varphi(0)$) by $L(\varphi)$.
Denote $L(G)$ the group generated by the linear part of $G$ some subgroup of analytic automorphisms $\varphi$ over $\C^n$ which fixes the origin $0$.
It is shown \cite[proof of Theorem 3.3]{Has93} that
for the covering transformation group $\pi_1(C)$ of a Hopf manifold $C$, $L(\pi_1(C))$, as an
automorphism group over $\C^n\setminus \{0\}$ is properly discontinuously without fixed point,
defining a compact quotient which is itself a Hopf manifold.
Let $T_t$, $(t \neq 0)$ be an analytic
automorphism over $\C^n\setminus \{0\}$ defined by
$$T_t(z_1, z_2, \cdots, z_n) = (tz_1, tz_2, \cdots, tz_n),$$
and set $g_t = T_tg T^{ -1}_t $, $G(t) = \{g_t | g \in \pi_1(X)\}$ and $G(0) = L(\pi_1(X))$. We define for $g \in \pi_1(X) $
an analytic automorphism $\tilde{g}$ over $(\C^n\setminus \{0\}) \times \C$
$$\tilde{g} : (z, t) \mapsto (g_t(z), t),$$
where $z = (z_1, z_2, \cdots, z_n) \in \C^n\setminus \{0\}$ and $t \in \C$. Then $\tilde{G} = \{\tilde{g} | g \in G\}$ is properly
discontinuous and fixed-point-free as an analytic automorphism group over $(\C^n\setminus \{0\}) \times \C$. The induced canonical map $(\C^n\setminus \{0\}) \times \C / \tilde{G} \to \C$ defines a deformation such that
the general fiber is $C$ while its center fiber is $(\C^n\setminus \{0\})/ L(\pi_1(C))$.
In fact, we can assume that $L(\pi_1(C))$ is of Jordan form such that the diagonal part is $\mathrm{diag}(\alpha_1, \cdots, \alpha_n)$. Continue to deform as follows. Let $S_t, (t  \neq 0)$ be an analytic automorphism over $\C^n\setminus \{0\}$ defined by
$$S_t(z_1, z_2, \cdots, z_n) = (t^{n-1}z_1, t^{n-2}z_2, \cdots,  z_n),$$
and set $L(\pi_1(C))(t) = \{S_t g S^{-1}_t | g \in L(\pi_1(C))\}$.
One can define a family
$(\C^n\setminus \{0\})/L(\pi_1(C))(t)$
as previous which
defines a deformation such that
the general fiber is $\C^n\setminus \{0\}/ L(\pi_1(C))$ while its center fiber is $(\C^n\setminus \{0\})/\mathrm{diag}(\alpha_1, \cdots, \alpha_n)$ as $S_t g S_t^{-1}$ has the effect of multiplying by $t$ over-diagonally.

\subsubsection{Stein nested coverings of Hopf manifolds}\label{nested-hopf}
In what follows, we construct  a family of nested coverings based open sets that are biholomorphic to an annulus times a polydisc.

Now fix a primary Hopf manifold $C$.
Let $\mathrm{diag}(\alpha_1, \cdots, \alpha_n)$ be the diagonal part of a generator of $L(\pi_1(C))$.
We choose the following covering of $(\C^n\setminus \{0\})/\mathrm{diag}(\alpha_1, \cdots, \alpha_n)$.
Take $r_i^j(1 \leq j \leq n, i=1,2,3,4)$ and $r_i^j >\delta >0$ such that for any $1 \leq j \leq n$
$$0 < r_1^j < r_2^j <r_3^j <r_4^j,$$
$$r_4^j=|\alpha_j|r_1^j,$$
and the domains $U_i^j(1 \leq j \leq n, i=1,2,3)$ in $\C^n$ defined by
$$U_i^j(\delta)=\{(z_1, \cdots, z_n) \in \C^n; r_i^j -\delta < |z_j|< r_i^j +\delta,|z_k|<r_4^j + \frac{ \delta}{2}, \forall k \neq j \}$$
satisfy
\begin{enumerate}
	\item for any $ i, j $, $U_i^j(\delta)$ is biholomorphic to its image under the quotient map $\pi: \C^n\setminus \{0\}\to (\C^n\setminus \{0\})/\mathrm{diag}(\alpha_1, \cdots, \alpha_n)$,
	\item for any $1 \leq j \leq n$,
	$$\pi(U_1^j(\delta)) \cap\pi(U_2^j(\delta)) \cap\pi(U_3^j(\delta)) = \emptyset. $$
\end{enumerate}

Consider the deformation in last paragraph.
The open sets
$U_i^j$ still give a Stein cover $\cU^\delta:=\{U_i^j(\delta)\}_{\delta}$ of sufficiently close deformation of $(\C^n\setminus \{0\})/\mathrm{diag}(\alpha_1, \cdots, \alpha_n)$
(i.e. a Stein cover of $C$).
See Step 1 in the proof of \cite[Theorem 2.8]{Tsu84}.
Note that $\cU^\delta$ defines a family of nested coverings.

\subsubsection{Existence of vector fields}
From now on, we will always choose this Stein nested covering of a primary Hopf manifold.
If $C$ is a primary Hopf manifold of diagonal type,
define $Z_i=z_i \frac{\d}{\d z_i}$ which descend to tangent vector fields on $C$.
The wedge product
$Z_1 \wedge \cdots \wedge Z_n$ is nowhere vanishing on the Shilov boundary of any open set in the Stein covering.
If $L(\pi_1(X))$ is Jordan matrix with diagonal equal to $\alpha$,
define $Z_i=\alpha z_i \frac{\d}{\d z_i}+z_{i+1} \frac{\d}{\d z_{i+1}}$ for $i \leq n-1$ and $Z_n=\alpha z_n \frac{\d}{\d z_n} $.
These tangent fields on $\C^n \setminus \{0\}$ descend to tangent vector fields on $C$.
The wedge product
$Z_1 \wedge \cdots \wedge Z_n$ is also nowhere vanishing on the Shilov boundary of any open set in the Stein covering.

\subsubsection{Line bundles and their cohomologies}
Let us study the cohomology of flat vector bundles over a primary Hopf manifold.
Note that we will study these cohomologies only over Hopf manifolds of generic type or classical type following Mall \cite{Mal91}. Computations of some cohomologies over non-primary Hopf manifolds were obtained in \cite{zhou-hopf,zhou-hopt-survey}.

Any flat line bundle $L$ over $C$ corresponds to a 1-dimensional representation of
$\pi_1(C)$.
Note that a primary Hopf manifold is diffeomorphic to $S^1 \times S^{2n-1}$.
Let $\beta$ be the image of the generator under the 1-dimensional representation of
$\pi_1(X)$. 
In the above choice of Stein nested covers,
let $l_{i_1,i_2}^j$ be
the transition function from $U_{i_2}^j(\delta)$ to $U_{i_1}^j(\delta)$.
We have for any $j$, $l_{31}^j=\beta$, $l_{13}^j=\beta^{-1}$
while all other transition functions are 1.
According to \cite[Theorem 4]{Mal91}, as any line bundle over a Hopf manifold is flat, the transition functions of $L$  satisfy the first conditions in Proposition \ref{tsu_lem} if the line bundle is not Hermitian flat.

We recall some results on the cohomology groups of flat line bundle over Hopf manifolds, following \cite{Mal91}.

First consider the condition on the vanishing of $H^1$.
Fix a Hopf manifold $C$.
Since $F,L$ are flat line bundles (corresponding to representations of $\pi_1(X)$),
for a deformation $p: \mathfrak{X} \to \Delta$ having $C$ as general fiber,
$F, L$ extend to flat line bundles $\mathfrak{F}, \mathfrak{L}$ over the total space $\mathfrak{X}$ (since the fundamental group is invariant under the deformation of complex structure).
By Grauert's semi continuity theorem,
to show that $h^1(X, F \otimes nL)=0$,
it is enough to show that
$h^1(\mathfrak{X}_0, \mathfrak{F} \otimes \mathfrak{L}|_{\mathfrak{X}_0})=0$.
In other words, it is enough to consider the Hopf manifolds of diagonal type.

The discussion up to now works for any Hopf manifold.
In the following, we need some cohomology calculation
studied by Mall \cite{Mal91}.
The surface case is completely studied in \cite[Section 3]{Mal91}.
(One main difference is that the divisors on Hopf surface are well understood (see e.g. \cite[Chap. V, (18.2)]{BHPV}).)


Then we have the following theorem of \cite[Theorem 2]{Mal91}, which implies the vanishing result.
\begin{myprop}
	Let $C$ be a Hopf manifold of generic type.
	Let $L$ be a flat line bundle over $C$ corresponding to $\beta \in \C^*$.
	Assume that $\beta \notin  \Delta_{\alpha_1, \cdots, \alpha_n}$.
	Then we have
	$$H^1(C, L)=H^0(C, L)=0.$$
\end{myprop}
If $C$ is a Hopf manifold of linear type, its tangent bundle $TC$ admits a filtration such that the graded pieces are Hermitian flat line bundles $L_{\alpha_i^{-1}}$ corresponding to constants  $\alpha_i^{-1}$.
Let $L$ be a holomorphic line bundle over $C$ corresponding to constant $\beta$.
To show that
$$H^1(C, T_C \otimes L)=H^0(C, T_C \otimes L)=0,$$
it is enough to study the corresponding vanishing of the flat line bundles corresponding to $\beta\alpha_i^{-1}$ by considering the long exact sequence of cohomology.
By the work of Mall, if $C$ is of generic type and
if $\beta\alpha_i^{-1} \notin  \Delta_{\alpha_1, \cdots, \alpha_n}$,
$$H^1(C, T_C \otimes L)=H^0(C, T_C \otimes L)=0.$$

Note that the cardinality of the possible $\beta$ such that
$\beta^{m}\alpha_i^{-1} \notin  \Delta_{\alpha_1, \cdots, \alpha_n}$
is at most countable.
In particular, for almost all flat line bundles $L$, for any $m \in \N$,
$$H^1(C, T_C \otimes L^{\otimes m})=H^0(C, T_C \otimes L^{\otimes m})=0.$$

\begin{myrem}
	Recalling Definition \ref{def-generic}, we note that the sufficient condition that for all $m \in \N$, $\beta^{m}\alpha_i^{-1} \notin  \Delta_{\alpha_1, \cdots, \alpha_n}$
	is equivalent to the condition that for all $m \in \N$, $\beta^{m} \notin  \Delta_{\alpha_1, \cdots, \alpha_n}$.
	In particular, under this sufficient condition, we have also that
	for any $m \in \N$,
	$$H^1(C,  L^{\otimes m})=H^0(C,  L^{\otimes m})=0.$$
	When $L$ is the normal bundle of $C$ as a hypersurface in a complex manifold, this sufficient condition implies the vanishing of the cohomological obstruction both for full linearization and for the Ueda problem.
\end{myrem}

We recall a result from \cite{GS21}~:
Throughout the paper $\|\cdot\|_D$ and $ |\cdot|_D$ denote respectively the $L^2$ and sup norms of a function in $D$, when $D$ is a domain in $\C^n$. If $E',E''$ are holomorphic vector bundles over $C$, we will fix a trivialization of $E'$ over $U_i$ by fixing a holomorphic basis $e'_k=\{e'_{k,1},\dots, e'_{k,m}\}$ in $\ov{U_k^{r^*}}$. We also fix a holomorphic base
	$e''_j=\{e''_{j,1},\dots, e''_{j,d}\}$ of $E''$ in $\ov{U_j^{r^*}}$. On $U_{ I}^{r^*}=U_{i_0}^{r^*}\cap\cdots\cap U^{r^*}_{i_q}$, it will be convenient to use the base
	\begin{equation*}
		e_ {i_0\dots i_q}:=e_{i_0}'\otimes e_{i_q}'':=\{e'_{i_0,k}\otimes e''_{i_q,j}\colon 1\leq k\leq m, 1\leq j\leq d\}.
\end{equation*}
Then we   define the $L^2$ norm for $f\in C ^q(\cL U^{r},\cL O(E'\otimes E''))$  by
\begin{eqnarray}\nonumber
	a_ Ie_{ I}&:=&\sum_{\mu=1}^{md}a_ I^{\mu} e_{ I,\mu},\\
	\|f\|_{\cL U^{r}}&:=&\max_{I=(i_0,\dots, i_q)\in  \cL I^ {q+1}, i_0 < \cdots < i_q}\left\{\|a_{ I}\circ
	\var_{i_q}^{-1}\|_{	\var_{i_q}(U_I)}\colon f_ i=a_{ I} e_{ I}\ \text{in $U_{ I}$}\right\}.
	\label{defnorm}\nonumber
\end{eqnarray}
\begin{mylem}\cite[Lemma A.2]{GS21} \label{SD-p} Let $\cL U^r=\{U_i^r\colon i\in \cL I\}$ with $r_*\leq r\leq r^*$ be  a family of nested finite coverings of $C$.
	Suppose that $f\in C^1(\cL U^{r^*},E'\otimes E'')$ and $f=0$ in $ H^1(\cL U^{r^*},E'\otimes E'')$. Assume that there is a solution $v\in C^{0}(\cL U^{r_*},E'\otimes E'')$  such that
	\begin{equation}\label{first-sol}
		\delta v=f, \quad
		\|v\|_{\cL U^{r_*}}\leq K\|f\|_{\cL U^{r^*}}.
	\end{equation}
	Then there exists a solution $u\in  C^{0}(\cL U^{r^*},E'\otimes E'')$ such that $\delta u=f$ on $\cL U^{r^*}$ and
	\begin{equation}\label{second-sol}
		\|u\|_{\cL U^{r^*}}\leq C( 
		|\{t'_{kj}\}|_{\cL U^{r^*}}+ K|\{t'_{kj}\}|_{\cL U^{r^*}}|\{t''_{kj}\}|_{\cL U^{r^*}})\|f\|_{\cL U^{r^*}},
	\end{equation}
	where $t_{kj}',t_{kj}''$ are the transition  
	matrices
	of $E',E''$, respectively, and $C$ depends only on the number  $|\cL I|$
	of open sets in $\cL U^{r^*}$ and transition functions of $C$. In particular, $C$ does not depend on $E',E''$.
\end{mylem}
It is well-known that one may calculate the \v{C}ech cohomologies by alternate \v{C}ech cochains (cf. e.g. \cite[(4.D), Chap. IV]{agbook}). By naturally identifying alternating \v{C}ech cochains with \v{C}ech cochains, we may replace the norm of the \v{C}ech cochains in the right-hand side of estimate (\ref{hopft-cohom-bound}) with the norm of the alternating \v{C}ech cochains $\| \bullet \|_{\cL U^{r}}$.
Note that in the left-handed side, for 0-cochains, alternating \v{C}ech cochains is equivalent to \v{C}ech cochains.


Note that in the original proof of \cite[Lemma A.2-Proposition A.4]{GS21}, the open sets in the cover $\cU^r$ are assumed to be biholomorphic to polydiscs.
But its proof works identically if the vector bundles $E', E''$ are locally trivial on these open sets.
\begin{mylem}\ Let $\cL U^r=\{U_i^r\colon i\in \cL I\}$ with $r_*\leq r\leq r^*$ be  the family of nested finite coverings of Hopf manifold $C$ as defined in section \ref{nested-hopf}. Let $F,L$ be two flat line bundles satisfying (\ref{lb1}). Let $m\in \N^*$
	Suppose that $f\in C^1(\cL U^{r^*},F\otimes L^{-m})$ and $f=0$ in $ H^1(\cL U^{r^*},F\otimes L^{-m})$. Assume that there is a solution $v\in C^{0}(\cL U^{r_*},F\otimes L^{-m})$  such that
	\begin{equation}\label{first-sol}
		\delta v=f, \quad
		\|v\|_{\cL U^{r_*}}\leq K\|f\|_{\cL U^{r^*}},
	\end{equation}
	where $K$ is independent of $m$. Then there exists a solution $u\in  C^{0}(\cL U^{r^*},F\otimes L^{-m})$ such that $\delta u=f$ on $\cL U^{r^*}$ and
	\begin{equation}\label{second-sol}
		\|u\|_{\cL U^{r^*}}\leq D\|f\|_{\cL U^{r^*}},
	\end{equation}
	$D$ does not depend on $m$.
\end{mylem}
\begin{proof}
The advantage to consider alternate \v{C}ech cochains appears as follows: From (\ref{lb1}), we have $|\{t''_{kj}\}|_{\cL U^{r^*}} \leq 1$ (up to an ordering of index). We then use Lemma \ref{SD-p} together with the extended definition of family of nested coverings.
\end{proof}

\subsection{Proof}

We will follow the proof of theorem \cite[Theorem 1.4]{GS21} in which the normal bundle is supposed to be unitary. Its proof uses the condition that the normal bundle is unitary only to apply \cite[Proposition 3.4 (d)]{GS21}. However, in our case of Hopf manifolds and our choice of open covers, the transition functions for alternate \v{C}ech 1-cochains can be choose of norm at most 1 (up to an ordering of index).
	Hence, we obtain the following {\it ad-hoc} version~:
	\begin{mythm}\label{lin-lesunit}
		Let  $C_n$ be a Hopf manifold of $X_{n+1}$. 
		Let $D$ be the constant appearing in (\ref{second-sol}) with line bundles $F:=T_C$ and $L:=N_{C|X}$. We can assumed that $D\geq 1$.
		Let $\eta_0=1$ and
		\begin{equation}\label{def-eta-foliation-int}\nonumber
			\eta_m
			:=D\max_{m_1+\cdots +m_p+s=m} \eta_{m_1}\cdots \eta_{m_p},
		\end{equation}
		where  the maximum is taken in $1\leq m_i<m$ for all $i$ and $s\in\Bbb N$. In particular, $\eta_m\leq D^m$.
		If $T_CM$ splits and  $ H^1(\cL U,T_CM\otimes N_{C|M}^{-\ell})=0$ for
		all $\ell>1$ or more generally if
		a neighborhood of $C$ in $M$ is 
		linearizable 
		by a formal holomorphic mapping which is tangent to the identity
		, then there exists a neighborhood of $C$ in $M$ which is holomorphically equivalent to a neighborhood of $C$ (i.e the $0$th section) in $N_C$.
		In that case, we say that the embedding $C\hookrightarrow M$ is holomorphically linearizable.
	\end{mythm}

\begin{proof}[Proof of Theorem \ref{hopf}]
	By assumption, we have
	$\beta \alpha_i^{ \pm 1} \notin  \Delta_{\alpha_1, \cdots, \alpha_n}$, $\beta^{- m} \alpha_i \notin  \Delta_{\alpha_1, \cdots, \alpha_n}$ for any $m \geq 1$.
	By the above discussion, we have that $T_X|_C=T_C \oplus N_{C/X}$ (i.e. the tangent bundle splits)
	since
	$$H^1(C, \Omega^1_C \otimes N_{C/X})=0.$$
	We also have for any $m \geq 1$,
	$$H^1(C, T_C \otimes N^{-m}_{C/X})=H^0(C, T_C \otimes N^{-m}_{C/X})=0.$$
	For any $m \geq 2$,
	$$H^1(C, T_X|_C \otimes N^{-m}_{C/X})=0.$$
	In particular, all cohomology conditions in Theorem \ref{lin-lesunit} are satisfied.
	%
	
	The conclusion follows from Theorem \ref{lin-lesunit}.
\end{proof}
\subsection{Other Hopf manifolds}
Another case where we may have the full linearization of Hopf manifolds
is the Hopf manifold of classical type.
\begin{mydef}
	A Hopf manifold $X$ (of dimension $>3$) is called of classical type if it is generated by a
	contraction of the type $\varphi: (z_1
	,\cdots,z_n) \to (\alpha z_1, \cdots, \alpha z_n)$ with $0
	< |\alpha | <1$.
\end{mydef}
Then we have the following theorem of \cite[Theorem 1]{Mal91}, which implies the vanishing result.
\begin{myprop}
	Let $X$ be a Hopf manifold of classical type.
	Let $L$ be a flat line bundle over $X$ corresponding to $\beta \in \C^*$.
	Assume that $\beta \notin  \Delta_{\alpha}$ (i.e. the free commutative subgroup of $\C^*$ generated by $\alpha$).
	Then we have
	$$H^1(X, L)=H^0(X, L)=0.$$
\end{myprop}
\begin{myrem}
	Theorem \ref{hopf} has an analog for a Hopf manifold of linear type which can deform to a Hopf manifold of classical type (by the same proof of Theorem \ref{hopf}).
\end{myrem}

The above arguments can also be generalized easily to the special case of non-primary Hopf manifolds studied in \cite{zhou-hopf}.
Recall the following cases of non-primary Hopf manifolds studied in \cite[Section 3]{zhou-hopf}.
\begin{mythm}(\cite[Proposition 1]{zhou-hopf})
	Let $X$ be Hopf manifolds of dimension $n > 2$ with
	$\pi_1 (X)$ identified with
	$\langle f, g \rangle \subset Aut(\C^n)$, and $f : (z_1 , \cdots , z_n ) \mapsto (\mu z_1 , \cdots , \mu z_n ),$  $g : (z_1 , \cdots , z_n ) \mapsto (a z_1 , \cdots , az_n ),$ with $0 < |\mu| < 1, a^m = 1$.
	Consider
	$L_{cd} $ a flat line bundle corresponding to the representation $\rho: \pi_1(X) \to \C^*$ with $\rho(f)=c$, $\rho(g)=d$.
	Assume that there exists no $r \in \N$ such that $c= \mu^r$, $d=a^r$.
	Then we have
	$$H^1(X, L)=H^0(X, L)=0.$$
\end{mythm}
\begin{mythm}(\cite[Proposition 2]{zhou-hopf})
	Let $X$ be generic Hopf manifolds of dimension $n > 2$ with
	$\pi_1 (X)$ identified with
	$\langle f, g \rangle \subset Aut(\C^n)$, and $f : (z_1 , \cdots , z_n ) \mapsto (\mu_1 z_1 , \cdots , \mu_n z_n ),$  $g : (z_1 , \cdots , z_n ) \mapsto (a z_1 , \cdots , az_n ),$ with $0 < |\mu_i| < 1 (\forall i \leq n), a^m = 1$.
	Assume that there are no relations of exponential between $\mu_1, \cdots, \mu_n$.
	Consider
	$L_{cd} $ a flat line bundle corresponding to the representation $\rho: \pi_1(X) \to \C^*$ with $\rho(f)=c$, $\rho(g)=d$.
	Assume that there exists no $v \in \Z^n$ such that $c= \mu^v$, $d=a^{|v|}$
	with $v \geq (1,\cdots,1)$ or $v \geq (0, \cdots,0)$ and some $v_i=0$.
	Then we have
	$$H^1(X, L)=H^0(X, L)=0.$$
\end{mythm}
Let $X$ be generic Hopf manifolds of dimension $n > 2$ with
$\pi_1 (X)$ identified with
$\langle f, g \rangle \subset Aut(\C^n)$, and $f : (z_1 , \cdots , z_n ) \mapsto (\mu_1 z_1 , \cdots , \mu_n z_n ),$  $g : (z_1 , \cdots , z_n ) \mapsto (a z_1 , \cdots , az_n ),$ with $0 < |\mu_i| < 1 (\forall i \leq n), a^m = 1$.
Take $r_i^j(1 \leq j \leq n, i=1,2,3,4)$ and $r_i^j >\delta >0$ such that for any $1 \leq j \leq n$
$$0 < r_1^j < r_2^j <r_3^j <r_4^j,$$
$$r_4^j=|\mu_j|r_1^j,$$
and the domains $U_i^{j,k}(1 \leq j \leq n, 1 \leq k \leq 2, i=1,2,3)$ in $\C^n$ defined by
$$U_i^{j,k}(\delta)=\{(z_1, \cdots, z_n) \in \C^n; r_i^j -\delta < |z_j|< r_i^j +\delta,
\frac{(k-1) \pi}{m}-\delta<\mathrm{arg}(z_j)< \frac{k \pi}{m}+\delta,
$$$$
|z_k|<r_4^j + \frac{ \delta}{2}, \forall k \neq j \}$$
with $\delta>0$ sufficiently small
satisfy
\begin{enumerate}
	\item for any $ i, j,k $, $U_i^{j,k}(\delta)$ is biholomorphic to its image under the quotient map $\pi: \C^n\setminus \{0\}\to X$,
	\item for any $1 \leq j \leq n$, $1 \leq k \leq 2$
	$$\pi(U_1^{j,k}(\delta)) \cap\pi(U_2^{j,k}(\delta)) \cap\pi(U_3^{j,k}(\delta)) = \emptyset. $$
\end{enumerate}
This defines a Stein nested covering of $X$.
In the above choice of Stein nested covering and notation $L_{cd}$,
let $l_{i_1,i_2}^{j,k_1,k_2}$ be
the transition function from $U_{i_2}^{j,k_1}(\delta)$ to $U_{i_1}^{j,k_2}(\delta)$.
We have for any $j$, $k_1$, $k_2$, $l_{31}^{j,k_1,k_2}=c d^{k_2-k_1}$, $l_{13}^{j,k_1,k_2}=c^{-1}d^{k_2-k_1}$
while all other transition functions are 1.

Define $Z_i=z_i \frac{\d}{\d z_i}$ which descend to tangent vector fields on $X$.
The wedge product
$Z_1 \wedge \cdots \wedge Z_n$ is nowhere vanishing on the Shilov boundary of any open set in the Stein covering.

Finally, for the convenience of the readers, we provide detailed proof of a result of X. Zhou.

\begin{myprop}\cite{Zhou04}.
 Let $L$ be a holomorphic line bundle over a non-primary Hopf manifold $X$.
 Then $L$ is flat.
\end{myprop}
\begin{proof}
First we choose a Galois cover $Y$ of $X$ (i.e. there exists a finite group $G \subset Aut(Y )$ such that the cover $\pi: Y \to X$ is
isomorphic to the quotient map $Y \to Y /
G$.)
It is well known that
given an arbitrary Hopf manifold $X$, there is a primary Hopf manifold $Y_1$
and a holomorphic finite cover $Y_1 \to X$.
The fundamental group $\pi_1(Y_1)$ is of finite index in $\pi_1(X)$.
This subgroup $\pi_1(Y_1)$ thus has a finite number
of conjugates in $\pi_1(X)$ and the intersection of these conjugates is a normal finite
index subgroup $G$ in $\pi_1(X)$.
The group $G$ corresponds to a holomorphic Galois cover $Y \to X$ with factorisation $Y \to Y_1 \to X$.
Since $Y_1$ is primary, as a finite cover, $Y$ is also primary.

Since $\pi^* L$ is a holomorphic line bundle over the primary Hopf manifold $Y$,
it is flat with a flat connection.
Since the cover $\pi$ is Galois, we have $L \simeq \pi^* L/G$.
Since $\pi$ is finite, one can construct a $G$-invariant flat
connection on $\pi^* L$ by averaging over the group action.
This induces a flat connection on $L$.
\end{proof}

\end{document}